\documentclass[a4paper,12pt]{amsart}
\usepackage{amsmath, amssymb, graphics}
\usepackage{a4wide}
\usepackage{amsmath}
\usepackage{amsthm}
\usepackage{amsfonts}
\usepackage{bbm}
\usepackage{graphics}
\usepackage{graphicx}
\usepackage{subfigure}
\usepackage{pstricks}
\usepackage{a4wide}
\usepackage{amsfonts}
\usepackage{abstract}
\usepackage{mathtools}
\usepackage{dsfont}
\usepackage[hidelinks]{hyperref}
\usepackage{setspace}
\newtheorem{innercustomthm}{Theorem}
\newenvironment{customthm}[1]
  {\renewcommand\theinnercustomthm{#1}\innercustomthm}
  {\endinnercustomthm}

\newcommand{\be}{\begin{eqnarray}}
\newcommand{\ee}{\end{eqnarray}}
\newcommand{\bea}{\begin{eqnarray}}
\newcommand{\eea}{\end{eqnarray}}

\newcommand{\p}{\partial}
\let\th=\theta
\let\w=\omega
\let\W=\Omega
\let\l=\labmda

\let\a=\alpha
\let\b=\beta
\let\g=\gamma
\let\G=\Gamma
\let\pa=\partial
\let\D=\Delta
\let\l=\lambda
\let\t=\tau
\newcommand{\wt}{{\omega}_{\mathsf{t}}}

\DeclareMathOperator{\Id}{Id}

\theoremstyle{plain}
\newtheorem{thm}{Theorem}
\newtheorem{prop}{Proposition}

\theoremstyle{definition}
\newtheorem{dfn}{Definition}
\newtheorem{ex}{Example}

\newtheorem{rem}[thm]{Remark}

\theoremstyle{remark}

\let\th=\theta

\DeclareMathOperator{\rk}{rk}

\numberwithin{equation}{section}

\begin{document}
%\title[The $\mathfrak{sl}_2(\mathbb{C})$ Lie algebra of quasi modular forms from VMHS]{ The $\mathfrak{sl}_2(\mathbb{C})$ Lie algebra of quasi modular forms from variation of mixed Hodge structures}
\title[Derivations of quasi modular forms from mirror symmetry]{The algebra of derivations of quasi modular forms from mirror symmetry}

\author{Murad Alim}
	%\address{FB Mathematik, Universit\"at Hamburg, Bundesstr. 55, 20146, Hamburg, Germany}
	\email{murad.alim@uni-hamburg.de}

\author{Vadym Kurylenko}
	%\address{FB Mathematik, Universit\"at Hamburg, Bundesstr. 55, 20146, Hamburg, Germany}
	\email{vadym.kurylenko@uni-hamburg.de}
	
	\author{Martin Vogrin}
	\email{martin.vogrin@uni-hamburg.de}

	\address{FB Mathematik, Universit\"at Hamburg, Bundesstr. 55, 20146, Hamburg, Germany}
	\maketitle
\begin{abstract}

We study moduli spaces of mirror non-compact Calabi-Yau threefolds enhanced with choices of differential forms. The differential forms are elements of the middle dimensional cohomology whose variation is described by a variation of mixed Hodge structures which is equipped with a flat Gauss-Manin connection. We construct graded differential rings of special functions on these moduli spaces and show that they contain rings of quasi-modular forms. We show that the algebra of derivations of quasi-modular forms can be obtained from the Gauss--Manin connection contracted with vector fields on the enhanced moduli spaces. We provide examples for this construction given by the mirrors of the canonical bundles of $\mathbb{P}^2$ and $\mathbb{F}_2$.

\end{abstract}
%\tableofcontents

\section{Introduction}

The study of mirror symmetry, which was discovered in the context of string theory, has had a remarkable impact on mathematics putting forward new connections between different research areas such as symplectic and algebraic geometry. In particular the field of enumerative geometry has benefited greatly from the study of mirror symmetry starting from the predictions of the numbers of rational curves of arbitrary degree on the quintic from the study of variation of Hodge structure of the mirror family \cite{Candelas:1991}, see also \cite{Cox:1999book,Vafa:2001school} and references therein.

A particularly fruitful arena of studies has been the interplay of modularity and geometry in the study of generating functions appearing in the enumerative geometry of families of Calabi-Yau (CY) manifolds. The quasi-modular forms studied by Kaneko and Zagier \cite{kaneko1995} appear for instance explicitly in the study of higher genus mirror symmetry of families of elliptic curves \cite{Dijkgraaf:1995elliptic}. The appearance of quasi-modular forms has moreover been a recurrent theme in the study of higher genus Gromov-Witten invariants for families of CY threefolds. 
The modularity of Gromov-Witten generating functions of CY families is related to the appearance of elliptic curves or K3 surfaces in the geometries. This can be either explicitly as factors of product spaces as for instance in \cite{Bershadsky:1994}, see also \cite{Marino:2003lecture} and references therein as well as \cite{oberdieck2016curve} for a recent treatment. Alternatively the CY families are elliptic or K3 fibrations as in e.~g.~ \cite{Hosono:1999ratell,Hosono:2003holanom,Klemm:2008FHSV,Alim:2012,Klemm:2012}. In the case of mirror symmetry of non-compact CY families the modularity can be traced to the appearance of mirror curves, a locus where the mirror families of canonical bundles over Fano manifolds degenerate. The latter modularity has been extensively studied, see e.~g.~\cite{Huang:2007holanomaly, Aganagic:2008, alim2014special, coates2018gromov}, as well as \cite{haghighat2017mirror} and references therein.

For higher genus mirror symmetry of generic families of CY threefolds a set of special functions on the moduli spaces which form a differential ring, i.~e.~are closed under derivations, was put forward for the mirror quintic by Yamaguchi and Yau \cite{Yamaguchi:2004} and generalized to arbitrary families in \cite{Alim:2007}. It was shown that the generating functions of higher genus Gromov-Witten invariants are polynomial in these special functions. An isomorphic differential ring of special functions for the mirror quintic was independently put forward by Movasati \cite{Movasati:20112} (see also \cite{Movasati:20111}),\footnote{The equivalence of the two constructions was shown in \cite{Alim:2016} and generalized to other families.} working on the moduli space of the mirror quintic enhanced with choices of differential forms. The motivation of Movasati's construction came from the study of the geometric realization of modular forms which we briefly outline in the following.

In \cite{Movasati:2008ell}, Movasati defined the moduli space $\mathsf{T}$ of pairs $(E,(\omega_1,\omega_2))$, where $E$ is an elliptic curve and $(\omega_1,\omega_2)$ is a basis of $H^1(E, \mathbb{C})$ such that $\w_1$ is holomorphic and the intersection pairing of $\w_1$ and $\w_2$ is 1.  It was observed that the ring of regular function $\mathcal{O}_\mathsf{T}$  is closely related to the ring of quasi-modular forms and can be equipped with a
differential structure coming from a special vector field on $\mathsf{T}$,  equivalent to the Ramanujan differential relations between the  Eisenstein series 
\[
     \pa_\t E_2 =\frac{E_2^2-E_4}{12},\quad \pa_\t E_4=\frac{E_4E_2-E_6}{3},\quad  \pa_\t E_6=\frac{E_6E_2-E_4^2}{2}\,, \quad \pa_\t:=\frac{1}{2\pi i} \frac{\partial}{\partial \tau}\,.
\]
Moreover, two other natural differential structures on $\mathcal{O}_{\mathsf{T}}$ were identified in \cite{Movasati:20113}, such that all together they constitute an  $\mathfrak{sl} _2(\mathbb{C})$ Lie algebra.  This Lie algebra is a realization of the $\mathfrak{sl}_2( \mathbb{C})$ triple of derivations acting on the ring of quasi-modular forms, see for instance \cite{zagier2008elliptic}. 

A similar construction can be carried out for more general varieties. Consider a  pair $(X,\omega)$, where $X$ is a compact Calabi-Yau variety  of dimension $n$ of fixed topological type, i.e. $X$ is diffeomorphic to some fixed variety $X_0$, and $\w$ is a basis of $H^n(X, \mathbb{C})$, respecting the Hodge filtration $\mathcal{F}^\bullet$ (see Definition \ref{dfn:enhco}) and such that the intersection pairing in this basis takes the form of a fixed constant matrix. Let $\mathsf{T}$ be a moduli space of such pairs. Analogously to the case of elliptic curves, there is a set of distinguished differential structures on the ring of regular functions $\mathcal{O}_\mathsf{T}$. They are defined by the special (modular) vector fields $\mathsf{R}_i$ on $\mathsf{T}$, defined by \be
\label{eq:mvf}
\nabla_{\mathsf{R}} \wt = \begin{pmatrix}
 0_{f_n,f_n} & *_{f_n,f_{n-1}} & 0_{f_n,f_{n-2}}& \dots &0_{f_n,f_0}\\
0_{f_{n-1},f_n} & 0_{f_{n-1},f_{n-1}} & *_{f_{n-1},f_{n-2}}& \dots &0_{f_{n-1},f_0}\\
\vdots&\vdots&\vdots&\vdots&\vdots\\
0_{f_1,f_n} & 0_{f_1,f_{n-1}}& 0_{f_1,f_{n-2}}&\dots&*_{f_1,f_0}\\
0_{f_0,f_n} & 0_{f_0,f_{n-1}} & 0_{f_0,f_{n-2}}& \dots &0_{f_0,f_0}
\end{pmatrix} \wt
%=A_\mathsf{R}\wt 
,\ee
where $\omega_{\mathsf{t}}$ is a frame of $\mathcal{H}^n(\mathsf{X} / \mathsf{T})$ defined in the section \ref{sectionenhanced} , $f_k = \rk \mathcal{F}^k / \mathcal{F}^{k+1}$ and $*_{a,b}$ is an $a \times b$ matrix with entries in $\mathcal{O}_\mathsf{T}$. 
%in Definition \ref{defmf}.

%characterized by the property  $\omega_{\mathsf{t}}$ ( 
%\be \label{eq:modvfi}
 %   \nabla_{\mathsf{R}_i}:\mathsf{F}^{k}/\mathsf{F}^{k+1}\rightarrow \mathsf{F}^{k-1}/\mathsf{F}^{k},
%\ee
%where $\mathsf{F}^k$ denotes the $k$-th Hodge filtration on the relative algebraic de Rham cohomology $\mathcal{H}^n(\mathsf{X}/\mathsf{T})$ and $\nabla$ is the Gauss-Manin connection on it. For a reference on the algebraic de Rham cohomology we refer to Grothendieck’s original article \cite{Grothendieck:1996algcoh}. The Gauss-Manin connection on it was introduced in \cite{Katz:1968}.
The construction of moduli spaces, together with the modular vector fields on them was carried out for a number of compact families: compact CY threefolds in \cite{Movasati:20112,Alim:2016,Alim:2014}, Dwork families in \cite{Movasati:2016,Nikdelan:2017},  genus 2 hyperelliptic curves in \cite{cao2019gauss}, K3 surfaces in \cite{Alim:2014,Nikdelan:2017}, elliptic K3 surfaces in \cite{Doran:2014,Alim:2018} and abelian varieties in \cite{fonseca2018higher}. The program of constructing moduli spaces of varieties enhanced with differential forms and investigating their properties was named Gauss-Manin Connection in Disguise (GMCD), see \cite{Movasati:2017}. 

Besides the (modular) vector field, there are other distinguished vector fields inducing differential structures on $\mathcal{O}_\mathsf{T}$.   Let $\mathsf{G}$ be the following  algebraic group 
\[\label{eq:lgr}
	\mathsf{G}=\{\mathsf{g}\in {\rm GL}( b_n,\mathbb{C})|~ \mathsf{g}~\mathrm{preserves~ filtration~and}~\mathsf{g}\Phi\mathsf{g}^{\rm tr}=\Phi\},
\]
where $b_n=\dim H^n(X, \mathbb{C})$ is the $n$-th Betti number, $\Phi$ is a constant skew-symmetric matrix, and let ${\rm Lie}(\mathsf{G})$ denote its Lie algebra. The special vector fields $\mathsf{R}_\mathfrak{g}$ on $\mathsf{T}$ can be defined as
\[
    \nabla_{\mathsf{R}_\mathfrak{g}}\omega_\mathsf{t}=\mathfrak{g}\omega_\mathsf{t},\qquad \mathfrak{g}\in{\rm Lie}(\mathsf{G}).
\]
Any vector field on $\mathsf{T}$ can be viewed as a derivation on $\mathcal{O}_{\mathsf{T}}$. We refer to the algebra of derivations coming from the vector fields ${\mathsf{R}_\mathfrak{g}}$ and the modular vector fields $\mathsf{R}_i$ as the Gauss-Manin Lie algebra.\footnote{Note that it has also been called AMSY Lie algebra in \cite{Nikdelan:2017,Nikdelan:2019}.} It is a a generalization of the algebra of derivations on the ring of quasi-modular forms to more general rings associated to CY varieties.  The algebra of derivations for elliptic curves was characterized in the framework of GMCD in \cite{Movasati:20113}.  I was computed for compact CY threefolds in \cite{Alim:2016}, for Dwork families in \cite{Movasati:2016,Nikdelan:2017} and for elliptic K3 surfaces in \cite{Alim:2018}. In \cite{Haghighat:2015} the framework of GMCD was studied in the case of elliptically fibered CY manifolds of various dimensions. In particular the authors of \cite{Haghighat:2015} identify an $\mathfrak{sl}_2( \mathbb{C})$ Lie subalgebra corresponding to the algebra of derivations of quasi-modular forms, which in their case is directly related to the modularity of the elliptic fiber. In the present work, we identify an $\mathfrak{sl}_2( \mathbb{C})$ Lie subalgebra of the Gauss--Manin Lie algebra which is more indirect and which is related to the modular structure of the mirror curves appearing in the study of mirror symmetry of non-compact CY threefolds.

This work is concerned with the study of the local structure of moduli spaces $\mathsf{T}$ for non-compact CY threefolds mirror to toric CY threefolds. These are of the form
\[ \label{ref:fami}
{\mathcal{X}} = \left\{   u v + \sum_{i=0}^{s-1}a_i y_i = 0  \; | ~ u, v \in \mathbb{C},~ y_i \in \mathbb{C}^{*}, \prod_{i=0}^{s-1} y_i^{Q^i_k} = 1  \right\}, \] 
where $Q^i_k$ is the charge matrix defining the toric CY, $s$ is the number of 1-dimensional cones in the corresponding fan and $k$ runs from $1$ to $s-3$.      
The middle cohomology of such varieties does not carry a pure Hodge structure, but instead -- a mixed Hodge structure. We construct locally a moduli space $\mathsf{T}$ of non-compact CY varieties enhanced with differential forms. By (not necessarily compact) CY threefold enhanced with differential forms we refer to the pair $(X,\omega)$, where $X$ is a CY threefold and $\w$ is a basis of its middle cohomology, respecting both the Hodge filtration $F^\bullet$ and the weight filtration $W_\bullet$ on $H^3(X,\mathbb{C})$. We further require that the pairing on $\operatorname{Gr}_i^W=W
_i/W_{i-1}$ takes the form of some fixed matrix $\Phi_i$ for all $i$. As for compact families, the ring of regular functions $\mathcal{O}_\mathsf{T}$ can be equipped with a set of derivations coming from the modular vector fields \eqref{eq:mvf}, as well as the derivations corresponding to elements in ${\rm Lie}(\mathsf{G})$, giving the full Gauss-Manin Lie algebra. 

The program is applied to two cases of interest, the mirror families to local $\mathbb{P}^2$ and local Hirzebruch surface  $\mathbb{F}_2$, as an illustration. The developed techniques are universal and can be generalized to other non-compact CY varieties, as well as variations of mixed Hodge structure associated to different geometries. For the mirrors of local $\mathbb{P}^2$ and $\mathbb{F}_2$, the local structure of the moduli spaces can partially be characterized by the inclusion of differential rings of the mirror curves into the differential ring of the non-compact CY threefold family. The following structural theorem gives the isomorphism of differential rings in the case of the mirror of local $\mathbb{P}^2$:
\begin{thm}\label{thm:1}
There is an isomorphism between the ring of regular functions $\mathcal{O}_{\mathsf{T}}$ for the mirror of local $\mathbb{P}^2$ equipped with the differential structure coming from a modular vector field $\mathsf{R}$ on $\mathsf{T}$ and the ring of quasi-modular forms on $\Gamma_0(3)$ with the derivation $\pa_\tau$. 
\end{thm}

For the mirror of local $\mathbb{F}_2$ we prove the following theorem, relating the differential ring of the family of non-compact CY threefold to the differential ring of the associated mirror curve:

\begin{thm}\label{thm:2}
The ring of regular functions $\mathcal{O}_{\mathsf{T}}$ for the mirror of local $\mathbb{F}_2$ equipped with the differential structures coming from modular vector fields $\mathsf{R}_1$ and $\mathsf{R}_2$ contains as a differential sub-ring the ring of quasi-modular forms on $\G_0(2)$ with the derivation $\pa_\tau$. 
\end{thm}

Moreover, the Gauss-Manin Lie algebra is computed in both cases and the following theorem characterize its relation to the Gauss-Manin Lie algebra of the associated mirror curve:
\begin{thm}\label{thm:3}
There is an $\mathfrak{sl}_2(\mathbb{C})$ Lie subalgebra of the Gauss-Manin Lie algebra $\mathfrak{G}$ for mirrors of local $\mathbb{P}^2$ and local $\mathbb{F}_2$.
\end{thm}
\subsection*{Organization of the paper}
This paper is organized as follows: In the second section we briefly review quasi-modular forms and construct the algebra of derivations on it. In the third section we review the theory of enhanced CY varieties and their moduli spaces. We recall in detail the construction for elliptic curves.  In the fourth section  we review the construction of variation of mixed Hodge structure for families of non-compact CY threefolds and propose a definition of enhanced CY varieties for non-compact situations, extending the GMCD program to the case of mixed Hodge structures. At the end of the article, we apply the construction to two families, mirror to non-compact CY threefolds and prove Theorems \ref{thm:1}, \ref{thm:2} and \ref{thm:3}.

\subsection*{Acknowledgements}
The authors would like to thank Florian Beck for helpful discussions. M.V. would like to thank IMPA, where part of this work was carried out, for hospitality. This research is supported by DFG Emmy-Noether grant on ”Building blocks of physical theories from the geometry of quantization and BPS states”, number AL 1407/2-1.

\section{Modular forms and differential operators}\label{sec:mf}
\subsection{Basic definitions}
In this section we review some basic definitions of quasi--modular forms and their associated algebra of derivations, see \cite{kaneko1995,Fred:2005course,zagier2008elliptic} for more background.\\
Let $\mathbb{H} = \{ \tau \in \mathbb{C} \: | \: \operatorname{Im}\tau >0 \}$ be the upper half plane and $\overline{\mathbb{H}}  = \mathbb{H} \cup \mathbb{Q} \cup \{i \infty\} $ be the extended upper half plane, which is closed under the natural ${\rm SL}_2(\mathbb{Z})$ action 
\[\t \rightarrow \gamma.\t=\frac{a\t +b}{c \t + d}, \qquad \text{for} \qquad \begin{pmatrix} a & b \\ c& d \end{pmatrix}=\gamma \in {\rm SL}_2(\mathbb{Z}).\] Let $N$ be a positive integer, then the \textit{principal congruence subgroup of level N} is
\[ \G(N) =  \left\{ \begin{pmatrix}
a& b\\
c& d
\end{pmatrix} \in {\rm SL}_2( \mathbb{Z}) : \begin{pmatrix} a & b \\ c& d \end{pmatrix} = \begin{pmatrix} 1 & 0 \\
0 & 1 \end{pmatrix}\mod N \right\}. \]
A subgroup $\G$ of ${\rm SL}_2( \mathbb{Z})$ is called  a \textit{congruence} subgroup if  $\G(N) \subseteq \G $. For example, 
\be\label{eq:csg} \G_0 (N) =  \left\{ \begin{pmatrix}
a& b\\
c& d
\end{pmatrix} \in {\rm SL}_2(\mathbb{Z}) : \begin{pmatrix} a & b \\ c& d \end{pmatrix} = \begin{pmatrix} * & * \\
0 & * \end{pmatrix}\mod N \right\}. \ee
\begin{dfn}
A \textit{multiplier system} is a map $v: {\rm SL}_2( \mathbb{Z}) \rightarrow \mathbb{C}^\times$, such that if  we denote 
\[ j_{v,k} (\gamma, z) : = v(\gamma) (c z + d)^k \]
for $\gamma \in {\rm SL}_2( \mathbb{Z})$, then
\[  j_{v,k} (\gamma_1 \gamma_2, z) = j_{v,k} (\gamma_1, z) j_{v,k} (\gamma_2, z) \] 
and 
\[ j_{v,k} (-\gamma, z) = (-1)^k j_{v,k} (\gamma, z) . \]
\end{dfn}
\begin{dfn}
  A \textit{modular form of weight} $k \in \mathbb{Z}$ with respect to a congruence subgroup $\G$ and a multiplier system $v$ is a function $f: {\mathbb{H}} \rightarrow \mathbb{P}^1 $  such that the following conditions hold:
\begin{enumerate}
\item $f(\g .\t) = v(\gamma) (c \t +d)^k  f(\t),\quad \forall \g = \begin{pmatrix}
a&b\\c&d
\end{pmatrix} \in \G$, where $(c \t +d)^k$ is called the automorphy factor,
\item $f$ is holomorphic on $\mathbb{H}$,
\item $f$ is holomorphic at the cusps in the sense that the function
\[ (c \t + d)^{-k} f(\g .\t) \]
is holomorphic at $\t = i \infty$ for any $\g \in \rm{PSL}_2( \mathbb{Z})= {\rm SL}_2(\mathbb{Z})/ \{\pm {\rm I} \} $. \end{enumerate}
\end{dfn}
 The ring of modular forms for $\G$ and $v$ is denoted by $\mathcal{M}(\G, v)$. It is naturally graded by the weight and finitely generated. If the multiplier system $v$ is trivial we simply omit it from notation. For example, 
\[
    \mathcal{M}\left({\rm SL}_2(\mathbb{Z})\right) = \bigoplus_{k\in2\mathbb{Z}} \mathcal{M}_k ({\rm SL}_2(\mathbb{Z})) \cong\mathbb{C}[E_4,E_6]
\]
is finitely generated by the Eisenstein series $E_4$ and $E_6$ defined by
\[
E_{2k}(\tau)=1+\frac{2}{\zeta(1-2k)}\sum_{n=1}^\infty \frac{n^{2k-1}q^n}{1-q^n},\quad {\rm for}\quad q=e^{2 \pi i \tau},
\]
where $\zeta$ is the Riemann $\zeta$-function. 

\begin{dfn}A \textit{quasi-modular form} on $\G$ of weight $k\in\mathbb{Z}$ and depth $p\in\mathbb{N}$ is a holomorphic function $f:  {\mathbb{H}} \rightarrow \mathbb{P}^1$, holomorphic at the cusps, such that there exist $p+1$ holomorphic functions $f_0, \dots, f_p$, such that 
\[ (c \t + d )^{-k} f (\g. \t) = \sum_{i=0}^p f_i(\t) \left( \frac{c}{c \t + d} \right)^i,\quad \forall \gamma \in \Gamma. \]
If a multiplier system is involved, one just has to multiply the right-hand side by $v(\g)$. 
\end{dfn}
\begin{ex}
The Eisenstein series $E_2$ is a quasi-modular form of depth 1 and weight 2 for $\G = {\rm SL}_2(\mathbb{Z})$
\[ E_2 (\g. \t) = (c \t + d)^2 E_2 ( \t) + \frac{12 c }{2 \pi i} (c \t + d). \]
\end{ex}
\noindent The ring of quasi-modular forms for ${\rm SL}_2(\mathbb{Z})$  is finitely generated
\[ \widetilde{\mathcal{M}}({\rm SL}_2(\mathbb{Z})) \cong \mathbb{C}[ E_2, E_4, E_6]. \] 
\subsection{$\mathfrak{sl}_2(\mathbb{C})$ Lie algebra of derivations} \label{sl2csection}
Unlike the ring of modular forms, the ring of quasi-modular forms is closed under the action of $\partial_\tau \coloneqq \frac{1}{2 \pi i} \frac{\pa}{\pa \t}$. For example, in the case of $\G = {\rm SL}_2(\mathbb{Z})$  we have the following relations
\begin{equation}
    \partial_\tau E_2=\frac{E_2^2-E_4}{12},\quad \partial_\tau E_4=\frac{E_4E_2-E_6}{3},\quad \partial_\tau E_6=\frac{E_6E_2-E_4^2}{2}, \label{ramanujan}
\end{equation}
called the Ramanujan relations.

In addition to the derivation $\pa_\t$ we can define two more derivations on $\widetilde{\mathcal{M}}({\rm SL}_2(\mathbb{Z}))$, as it was done in \cite{zagier2008elliptic, zagier2016partitions}.  Due to the fact that $\widetilde{\mathcal{M}}({\rm SL}_2(\mathbb{Z})) \cong \mathbb{C}[ E_2, E_4, E_6]$, we can write any quasi-modular form $f$ of weight $k$ as a polynomial in $E_2$
\[ f = P_f( E_2) = \sum_i a_n E_2^n,  \]
with $a_n\in\mathcal{M}({\rm SL}_2(\mathbb{Z}))$ of weight $k-2n$. Define 
\[ \mathfrak{d} f = - 12 P_f'(E_2). \] 
The other operator  simply returns the weight of the quasi-modular form
\[ W(f) = k f. \]

%\begin{align*}
%    \pa_\tau f( E_2, E_4, E_6) =& \left( \frac{1}{12}(E_2^2-E_4) \pa_{E_2} + \frac{1}{3} (E_4E_2-E_6) \pa_{E_4} + \frac{1}{2}(E_6E_2-E_4^2) \pa_{E_6} \right) f( E_2, E_4, E_6), \\
%W f( E_2, E_4, E_6) =& \left( 2 E_2 \pa_{E_2} + 4 E_4 \pa_{E_4} + 6 E_6 \pa_{E_6} \right) f( E_2, E_4, E_6), \\
% \mathfrak{d} f (E_2, E_4, E_6 )=&-12\partial_{E_2}f(E_2, E_4, E_6).
%\end{align*}  

\begin{prop}
The operators $\pa_\t$, $\mathfrak{d}$ and $W$ defined above form  an $\mathfrak{sl}_2(\mathbb{C})$ Lie algebra 
\be\label{eq:sl2c} [\: W, \pa_\t \:] = 2 \pa_\t, \qquad [\:W, \mathfrak{d} \:] = - 2 \mathfrak{d}, \qquad [\:\pa_\t,\mathfrak{d} \: ] = W.  \ee
It is possible to construct this algebra of derivations for the other congruence subgroups as well (see section 5.3 in \cite{zagier2008elliptic}). 
\end{prop}
\begin{rem}
There exists a similar Lie algebra of derivations in the case of Jacobi modular forms, see \cite{oberdieck2019gromov}. 
\end{rem}

\subsection{Differential rings of quasi-modular forms}
\label{sec:app} 
In this section we review the differential rings of quasi-modular forms described in \cite{maier2} (see also \cite{alim2014special}) for the congruence subgroups $\Gamma_0(N)$  with $N=1^*$,\footnote{Here $1^*$ corresponds to $\G_0(1)^*$ -- the unique normal index 2 subgroup of $\G(1)$.} $2,\:3$. For suitable multiplier systems on $\Gamma_0(N)$ (see loc.cit.) consider the following weight one modular forms:
\begin{table}[h]\label{tab:tab3}
	\centering
	\begin{tabular}{c|c  c  c  }
		$N$&$A$&$B$&$C$ \\
		$1^*$&$E_4(\tau)^{1/4}$&$\left(\frac{E_4(\tau)^{3/2}+E_6(\tau)}{2}\right)^{1/6}$&$\left(\frac{E_4(\tau)^{3/2}-E_6(\tau)}{2}\right)^{1/6}$ \\
		2&$\frac{(64\eta(2\tau)^{24}+\eta(\tau)^{24})^{1/4}}{\eta(\tau)^2\eta(2\tau)^2}$&$\frac{\eta(\tau)^4}{\eta(2\tau)^2}$&$2^{3/2}\frac{\eta(2\tau)^4}{\eta(\tau)^2}$ \\
		3&$\frac{(27\eta(3\tau)^{12}+\eta(\tau)^{12})^{1/3}}{\eta(\tau)\eta(3\tau)}$&$\frac{\eta(\tau)^3}{\eta(3\tau)}$&$3\frac{\eta(3\tau)^3}{\eta(\tau)}$ 
	\end{tabular}
\end{table}\\
\noindent where $\eta(\tau)$ denotes the Dedekind $\eta$-function
\[
    \eta(\tau)=\left(\frac{E_4(\tau)^3-E_6(\tau)^2}{1728}\right)^{\frac{1}{24}}.
\]
Although some of these functions are not single-valued, their squares are. Moreover, the modular forms $B$ and $C$ can be written in terms of $A$ and a weight $0$ modular form $\a$ - the Hauptmodul. The following holds: 
\[ B = (1-\a)^{1/r} A , \quad C = \a^{1/r} A ,\] 
where $r=6$ for $N=1^*$, $r=4$ for $N=2$ and $r=3$ for $N=3$.
%\begin{table}[htb]
 %   \centering
  %  \begin{tabular}{c|c c c }
  %    $N $  & $1^*$ & $2$ & $3 $ \\
  %     $\a$  & ? & $2^{12} \frac{\eta(2 \tau)^24}{\eta(\tau)^24}$ & $3^6 \frac{\eta(3 \tau)^12}{\eta(\tau)^12}$ 
%    \end{tabular}
%\end{table}

Define also the analogue of the Eisenstein series $E_2$ as
\begin{equation}\label{eq:ee}
E=\partial_\tau\log B^rC^r. 
\end{equation}

\begin{dfn}
By a differential ring of quasi-modular forms on $\Gamma_0(N)$ we would mean a polynomial ring modulo the above relations: 
\[ R_N = \mathbb{C}[ \a, A, B, C,E] / (B^r - (1- \a) A^r, C^r - \a A^r) \] 
with differential relations given by 
\begin{equation}\label{eq:ringmf}
\begin{split}
\pa_\tau \a& = \a( 1- \a) A^2,\\
\partial_\tau A&=\frac{1}{2r}A\left(E+\frac{C^r-B^r}{A^{r-2}}\right),\\
\partial_\tau B&=\frac{1}{2r}B(E-A^2),\\
\partial_\tau C&=\frac{1}{2r}C(E+A^2),\\
\partial_\tau E&=\frac{1}{2r}(E^2-A^4).
\end{split}
\end{equation}
\end{dfn}
The other derivations on the ring are given by
\begin{equation}
    H: f\mapsto kf,
\end{equation}
and
\begin{equation}
    F: f\mapsto -2r\partial_{E}P_f(E),
\end{equation}
for a weight $k$ quasi-modular form $f\in\widetilde{\mathcal{M}}(\Gamma_0(N))$, where $P_f(E)$ is viewed as a formal expansion of $f$ in $E$, as in the section \ref{sec:mf}. The Lie algebra of derivations is $\mathfrak{sl}_2(\mathbb{C})$.
\section{Moduli space of enhanced CY varieties}\label{sec:cpct} 

\subsection{Hodge structures}\label{sec:32}
\begin{dfn}
A \textit{pure $\mathbb{Z}$-Hodge structure} of weight $n\in\mathbb{Z}$ is a pair $(H_{\mathbb{Z}}, F^\bullet)$, where  $H_{\mathbb{Z}}$ is a lattice, i.~e.~a finitely generated free abelian group, and $F^\bullet$ is a decreasing (Hodge) filtration on $H_\mathbb{C} = H_\mathbb{Z} \otimes_\mathbb{Z} \mathbb{C}$
\[ \label{eq:filt} 0 \subset F^{n} \subset \dots \subset F^{0} = H_\mathbb{C}, \]
such that $F^p \oplus \overline{F^{n-p+1}} \simeq H_{\mathbb{C}}$. 
\end{dfn}
\begin{dfn}
A \textit{polarization} of a $\mathbb{Z}$-Hodge structure $( H_{\mathbb{Z}}, F^\bullet)$ is a bilinear form 
\[ \langle -, - \rangle: H_\mathbb{Z} \otimes H_\mathbb{Z} \rightarrow \mathbb{Z},  \] 
such that:
\begin{itemize}
    \item It is symmetric if $n$ is even and skew-symmetric if $n$ is odd,
    \item The orthogonal complement of $F^p$ is $F^{n-p+1}$,
    \item For $\a, \b\in H_\mathbb{C}$ the hermitian form $\langle C \a , \overline{\b}\rangle$ is positive definite, where $C$ is the so-called Weil operator which  acts on $H^{p,q}\coloneqq F^p \cap \overline{F^q}$  by multiplication with $i^{p-q}$.
\end{itemize}
A Hodge structure is called \textit{polarizable} if it admits a polarization.
\end{dfn}

\begin{dfn}
	A \textit{variation of  Hodge structure} on a complex variety $\mathsf{B}$ is a pair $(\mathcal{H}_\mathbb{Z},\mathcal{F}^\bullet)$ such that
	\begin{itemize}
		\item $\mathcal{H}_\mathbb{Z}$ is a locally constant sheaf of finitely generated $\mathbb{Z}$-modules on $\mathsf{B}$,
		\item $\mathcal{F}^\bullet$ is a finite decreasing  filtration on $\mathcal{H}=\mathcal{H}_\mathbb{Z}\otimes_\mathbb{Z}\mathcal{O}_\mathsf{B}$ by holomorphic subbundles,
	\end{itemize}
	such that
	\begin{itemize}
		\item For each $b\in \mathsf{B}$ the fibers $\mathcal{F}^\bullet_b$  form a Hodge filtration on  $(\mathcal{H}_{\mathbb{Z}})_b \otimes_{\mathbb{Z}} \mathbb{C}$, 
		\item The Gauss-Manin conection $\nabla: \mathcal{H} \rightarrow \mathcal{H}\otimes \W^1_\mathsf{B}$  defined by 
		\[ \nabla(s\otimes f)=s\otimes df,\]
		for $s\in \mathcal{H}_\mathbb{Z}$ and $f\in \mathcal{O}_\mathsf{B}$, satisfies the Griffiths' transversality condition
		\[
\nabla \mathcal{F}^p\subset \mathcal{F}^{p-1}\otimes_{\mathcal{O}_{\mathsf{B}}}\Omega^1_{\mathsf{B}}. \] 
	\end{itemize}
	A variation of Hodge structure is said to be \textit{polarizable} if there exists a $\nabla$-flat bilinear pairing 
		\[\langle-,-\rangle: \mathcal{H}_\mathbb{Z}\otimes \mathcal{H}_\mathbb{Z} \rightarrow \mathbb{Z},\] 
		such that it polarizes $(\mathcal{H}_{\mathbb{Z}})_b$  for each  $b \in \mathsf{B}$. 
\end{dfn}
\begin{dfn}
A \textit{Calabi-Yau variety} is a non-singular quasi-projective variety $X$, such that  the canonical sheaf $ \w_X$ is trivial and $H^i(X, \mathcal{O}_X) =0$ for $0<i< \dim X$.
\end{dfn}

\begin{ex} \label{ex:ccy}
Let $\pi: \mathcal{X}\rightarrow \mathsf{B}$ be a family of compact CY varieties of dimension $n$. The middle cohomology of a  general fibre $X_b=\pi^{-1}(b),~b\in\mathsf{B}$ carries a Hodge structure $(H^n(X_b,\mathbb{Z}),F^\bullet(X_b))$, naturally polarized by
\begin{equation}\label{eq:pairp}
    \langle\alpha,\beta\rangle=(-1)^{\frac{n(n-1)}{2}}\int_{X_b} \alpha\wedge\beta\qquad \alpha,\beta\in H^n(X_b,\mathbb{Z}),
\end{equation}
and there is a polarized variation of Hodge structure $(\mathcal{H}^n_\mathbb{Z}(\mathcal{X}/\mathsf{B}),\mathcal{F}^\bullet)$, where $\mathcal{H}^n_\mathbb{Z}(\mathcal{X}/\mathsf{B})=R^n\pi_*\mathbb{Z}$, and $\mathcal{F}^\bullet$ are unique vector subbundles of $\mathcal{H}^n(\mathcal{X}/\mathsf{B}) = \mathcal{H}^n_\mathbb{Z}(\mathcal{X}/\mathsf{B}) \otimes_\mathbb{Z} \mathcal{O}_\mathsf{B}$ such that $\mathcal{F}^\bullet_b=F^\bullet(X_b)$. The variation of Hodge structure is naturally polarized by \eqref{eq:pairp}.
\end{ex}

\subsection{Enhanced families} \label{sectionenhanced}
\begin{dfn} \label{dfn:enhco}
An \textit{enhanced compact CY variety} of dimension $n$ is a pair $(X, \w)$, where $X$ is a compact CY variety of dimension $n$ and
 \[ \w = ( \w_1,\dots ,  \w_i,\dots, \w_{{b_n}})^{\rm tr}, \quad b_n =\dim H^n(X, \mathbb{C}), \]
is  a basis of $H^n(X,\mathbb{C})$ compatible with the Hodge filtration, i.~e. 
\[ (\w_1, \ldots, \w_{\dim F^k}) \]
spans a basis of ${F}^{k}$, and such that 
\[ \langle\omega,\omega\rangle=\Phi, \]
for a constant non-degenerate matrix $\Phi$. We call a basis $\w$ satisfying the above conditions \textit{compatible}. 
\end{dfn}

%We will denote by $\mathsf{T}$ the moduli space of such pairs, and by $\mathsf{X} / \mathsf{T}$ the associated family of enhanced CY varieties over it. 

In this paper we focus on the local structure of the moduli space $\mathsf{T}$ rather than its global structure which requires a more careful study. We will denote by $\mathsf{X} / \mathsf{T}$ the associated family of enhanced CY varieties over it. For construction of the moduli space $\mathsf{T}$ in some specific examples we refer for example to \cite{movasati2020modular}. Consider a family $\mathcal{X}/\mathsf{B}$ of smooth CY varieties as in Example \ref{ex:ccy}. Around any point $b\in \mathsf{B}$ we can choose a Zariski open neighborhood $V_b \subset \mathsf{B}$ and a global frame $\w_b$ of the vector bundle $\mathcal{H}^n(\mathcal{X}/\mathsf{B}) \big|_{V_b}$, such that at every point $b'$ of $V_b$ the basis $\w_b(b')$ is compatible. All other frames with this property can be constructed by a transformation of $\omega_b$ by a non-singular lower block-triangular $b_n\times b_n$ matrix 
\be S = \begin{pmatrix}\label{eq:ss}
 *_{f_n,f_n} & 0_{f_{n},f_{n-1}} & 0_{f_{n},f_{n-2}}& \dots &0_{f_{n},f_{0}}\\
 *_{f_{n-1},f_{n}} & *_{f_{n-1},f_{n-1}} & 0_{f_{n-1},f_{n-2}}& \dots &0_{f_{n-1},f_{0}}\\
 \vdots&\vdots&\vdots&\vdots&\vdots\\
 *_{f_{1},f_{n}} & *_{f_{1},f_{n-1}} & *_{f_{1},f_{n-2}}& \dots &0_{f_{1},f_{0}}\\
 *_{f_{0},f_{n}} & *_{f_{0},f_{n-1}} & *_{f_{0},f_{n-2}}& \dots &*_{f_{0},f_{0}}
\end{pmatrix} ,\ee
where $f_k = \rk \mathcal{F}^k / \mathcal{F}^{k+1}$ and  $*_{a,b}$ is an $a \times b$ matrix, satisfying 
\be\label{paircond} S \Phi S^{\rm tr} = \Phi.  \ee
We denote the corresponding frame by $\wt = S \w_b$. Define  $U$ to be the underlying complex variety of 
\[ \operatorname{Spec} \left( \mathbb{C}\left[ s_{ij}^{ind}, \frac{1}{\det S}\right]\right),  \] 
where $s_{ij}^{ind}$ are the algebraically independent entries of the matrix $S$, then a patch $\tilde{\mathsf{T}}$ of the moduli space $\mathsf{T}$ can be constructed as the fiber product 
\[ \tilde{\mathsf{T}} = V_b \times_\mathbb{C} U. \]
Since we are going to work only with structures that do not require a construction of the global moduli space $\mathsf{T}$, we choose an analytic neighborhood $V_0$ inside of $V_{b_{0}}$ and an analytic neighborhood $U_0 \subseteq U$ of the point $S = \Id$.  Further in this paper we abuse notation and by $\mathsf{T}$ we mean \[ \mathsf{T} = V_0 \times_\mathbb{C} U_0. \]

 On the algebraic de Rham cohomology $\mathcal{H}^n(\mathsf{X} / \mathsf{T})$ there exists an algebraic Gauss-Manin connection \cite{Katz:1968} 
\be \nabla: \mathcal{H}^n(\mathsf{X} / \mathsf{T}) \rightarrow \mathcal{H}^n(\mathsf{X} / \mathsf{T}) \otimes_{\mathcal{O}_\mathsf{T}} \Omega^1_{\mathsf{T}}, \label{gmc} \ee 
where $\mathcal{O}_\mathsf{T}$ is the ring of regular functions and $\Omega^1_{\mathsf{T}}$ is an $\mathcal{O}_\mathsf{T}$-module of differential 1-forms on $\mathsf{T}$. Using the local description of $\mathsf{T}$, we can calculate the Gauss-Manin connection \eqref{gmc} from the Gauss-Manin connection on $\mathcal{H}^n(\mathcal{X}/\mathsf{B})$. Let 
\[\nabla \w_b = A \w_b,\]
with an $\W_{\mathsf{B}}^1$-valued matrix $A$, then in the basis $\wt$ we have 
\be\label{eq:mvfco} \nabla \wt = (d S + S A) S^{-1} \wt. \ee

\subsection{Gauss-Manin Lie algebra} \label{GMA}
It was shown in \cite{Movasati:20113} that in the case of elliptic curves  the ring of regular functions $\mathcal{O}_{\mathsf{T}}$ equipped with a differential structure coming from a certain vector field $\mathsf{R}$ on $\mathsf{T}$ is isomorphic, as a differential ring, to the ring of quasi-modular forms of \cite{kaneko1995} with derivation $\pa_{\tau}$. In other words, the Ramanujan differential equations \eqref{ramanujan}  are equivalent to the existence of a (modular) vector field $\mathsf{R}$ on $\mathsf{T}$. There exists a similar algebraic structure on the rings of regular functions $\mathcal{O}_\mathsf{T}$ for more general varieties. It can be formalized by introducing special (modular) vector fields on $\mathsf{T}$.

\begin{dfn} \label{defmf}
A \textit{modular vector field} $\mathsf{R}$ is a rational vector field on $\mathsf{T}$ such that
%\be \nabla_{\mathsf{R}} : \mathsf{F}^k/\mathsf{F}^{k+1} \rightarrow \mathsf{F}^{k-1}/\mathsf{F}^{k},  \label{filtr}\ee
% where $\mathsf{F}^\bullet$ is a filtration on $\mathcal{H}^n(\mathsf{X}/\mathsf{T})$ induced by the filtration $\mathcal{F}^\bullet$ on $\mathcal{H}^n(\mathcal{X}/\mathsf{B}) $.  Hence $\mathsf{R}$ is of the form
\be
%\label{eq:mvf}
\nabla_{\mathsf{R}} \wt = \begin{pmatrix}
 0_{f_n,f_n} & *_{f_n,f_{n-1}} & 0_{f_n,f_{n-2}}& \dots &0_{f_n,f_0}\\
0_{f_{n-1},f_n} & 0_{f_{n-1},f_{n-1}} & *_{f_{n-1},f_{n-2}}& \dots &0_{f_{n-1},f_0}\\
\vdots&\vdots&\vdots&\vdots&\vdots\\
0_{f_1,f_n} & 0_{f_1,f_{n-1}}& 0_{f_1,f_{n-2}}&\dots&*_{f_1,f_0}\\
0_{f_0,f_n} & 0_{f_0,f_{n-1}} & 0_{f_0,f_{n-2}}& \dots &0_{f_0,f_0}
\end{pmatrix} \wt
%=A_\mathsf{R}\wt 
,\ee
where $f_k = \rk \mathcal{F}^k / \mathcal{F}^{k+1}$ and $*_{a,b}$ is an $a \times b$ matrix with entries in $\mathcal{O}_\mathsf{T}$. 
\end{dfn}

Moreover, it is possible to define a set of vector fields, corresponding to derivations on $\mathcal{O}_\mathsf{T}$, generalizing \eqref{eq:sl2c}.

Let $\mathsf{G}$ be the algebraic group  consisting of all matrices $S$ defined above, 
    \[\label{eq:lg}
	\mathsf{G}=\{\mathsf{g}\in {\rm GL}(b_n,\mathbb{C})|~ \mathsf{g}~\mathrm{is~of~the~form~\eqref{eq:ss}~and}~\mathsf{g}\Phi\mathsf{g}^{\rm tr}=\Phi\}.
	\]
\begin{thm}\cite{Alim:2016}
For every element $\mathfrak{g} \in \operatorname{Lie}(\mathsf{G})$ there is a unique vector field $\mathsf{R}_{\mathfrak{g}}$ on $\mathsf{T}$ such that 
\[ \nabla_{\mathsf{R}_{\mathfrak{g}}} \wt = \mathfrak{g} \wt. \] 
\end{thm}

\begin{dfn}
%version2: The $\mathcal{O}_{\mathsf{T}}$-module generated by the vector fields $\mathsf{R}_{\mathfrak{g}}$ and the modular vector fields is closed under taking brackets, but it is not necessarily a Lie algebra. In the case when it is true, we call it \textit{Gauss-Manin Lie algebra}. 
The \textit{Gauss-Manin Lie algebra} $\mathfrak{G}$ is an $\mathcal{O}_{\mathsf{T}}$-module generated by the vector fields $\mathsf{R}_{\mathfrak{g}}$ and the modular vector fields.
\end{dfn}

\subsection{Elliptic curve} We review the case of a family of elliptic curves following the work of \cite{Movasati:20113}. Consider an affine variety in the weighted projective space $\mathbb{P}(1,2,3)$ given by
\be \label{eq:ellfam} \left\{W=x_1^6 + x_2^3  + x_3^2 - z^{-1/6} x_1 x_2 x_3 = 0 \right\}, \qquad z\in\mathbb{P}^1 \setminus \{ 0, \frac{1}{432}, \infty\},  \ee
with $x_i$ being the homogeneous coordinates on $\mathbb{P}(1,2,3)$.  There is an action of the finite group $\mathbb{Z}_2 \times \mathbb{Z}_3$ on it given by  \[\mathbb{Z}_2 : (x_1, x_2, x_3) \rightarrow ( \mu_2 x_1, x_2, \mu_2 x_3),\quad
\mathbb{Z}_3 : ( x_1, x_2,  x_3) \rightarrow ( \mu_3^2 x_1, \mu_3 x_2, x_3),\]
where $\mu_i$ is an $i$-th root of unity. A resolution of the quotient family gives rise to a family of smooth elliptic curves with a general fiber $E_z$. The monodromy group of this family is $\G_0(1)^*$ - the unique index two subgroup of ${\rm SL}_2(\mathbb{Z})$. The 
periods of the holomorphic differential 
\[
    \Omega={\rm Res}_{W=0}\left(\frac{-x_1dx_2\wedge dx_3+2x_2dx_1\wedge dx_3-3x_3dx_1\wedge dx_2}{W}\right),
\]
are annihilated by the following Picard-Fuchs operator
\[
 \th^2 - 12 z ( 6 \th + 5) ( 6 \th + 1), \qquad \th = z \frac{d}{d z}.  \]
Fix  a basis
\[ \w = ( \Omega, (1 - 432 z) \nabla_\th \Omega )^{\rm tr}, \]
of $H^1(E_z, \mathbb{C})$. It respects the Hodge filtration, i.~e.~$\Omega \in F^1, (1 - 432 z) \nabla_\th \Omega \in F^0$ and the pairing is 
 \[  \int_{E_z} \Omega \wedge (1-432 z) \nabla_\th \Omega  = -1. \]
In other words, the pairing matrix is
\[ 
\Phi = \langle \w_i, w_j \rangle = \begin{pmatrix}
 0 & -1 \\
 1 & 0 
\end{pmatrix},
\]
in this basis. The tranformations of the basis preserving filtration condition are of the form 
\[ S = \begin{pmatrix}
    s_{11} & 0 \\
    s_{21} & s_{22} 
  \end{pmatrix}.
  \]
The pairing condition \eqref{paircond} implies 
\[ S = \begin{pmatrix}
s^{-1}_{22} & 0 \\
s_{21} & s_{22} 
\end{pmatrix}.
\]

Therefore the moduli space $\mathsf{T}$ of enhanced elliptic curves is three-dimensional and we use a coordinate patch $\{ z, s_{21}, s_{22} \}$. 
 
% The special geometry basis is $\Omega^{SG} = $ and the flat coordinate is $\tau = $, we have 
 %\be \nabla_{\frac{\pa}{\pa \tau}} \Omega^{SG} = \begin{pmatrix}
 % 0 & 1\\
  %0 &0
 %\end{pmatrix} \Omega^{SG}. \ee 

 \begin{prop}
 There is a unique vector field $\mathsf{R}$ on $\mathsf{T}$ , such that
 \[ \nabla_{\mathsf{R}} \wt = \begin{pmatrix}
  0 & 1\\
  0 &0
 \end{pmatrix} \wt. \]
 It is given by 
 \[ \mathsf{R}= (1 - 432 z) z s_{22}^2 \frac{\pa}{\pa z} - s_{21} s_{22}^2 \frac{\pa}{\pa s_{22}}  - 60 z (1 - 432 z) s_{22}^3 \frac{\pa}{\pa s_{21}}.  \]
 \end{prop}
 \begin{proof}
 Existence and uniqness follow from a direct computation. 
 \end{proof}

We consider the group 
\[ \mathsf{G} = \left\{S= \begin{pmatrix}
    s^{-1}_{22} & 0 \\
    s_{21} & s_{22} 
  \end{pmatrix}
   \;|\; s_{22} \in \mathbb{C}^*, s_{21} \in \mathbb{C} \right\}. \]
  Its Lie algebra is easily computable and has two generators 
  \[ \mathfrak{g}_{22} = \left. \left(\frac{\pa}{\pa s_{22}} S \right) S^{-1} \right|_{S=\operatorname{Id}}= \begin{pmatrix}
   -1 & 0 \\
   0 & 1 
  \end{pmatrix}, \]
  \[ \mathfrak{g}_{21} =\left. \left(\frac{\pa}{\pa s_{21}} S \right) S^{-1} \right|_{S=\operatorname{Id}}= \begin{pmatrix}
   0 & 0 \\
   1 & 0 
  \end{pmatrix}. \]
 To find the corresponding vector fields, we look for solutions of the system 
 \[ \left(i_{1*} \nabla_{\frac{\pa}{\pa z}} + i_{2*} \nabla_{\frac{\pa}{\pa s_{22}}} + i_{3*} \nabla_{\frac{\pa}{\pa s_{21}}} \right) \wt = \mathfrak{g}_{2*} \wt, \] 
 of four linear equations for three variables $i_{1*}, i_{2*}, i_{3*}$. Note, that if there is a solution it is unique. We find the following: 
 \[ \mathsf{R}_{\mathfrak{g}_{22}} = s_{22} \frac{\pa}{\pa s_{22}} + s_{21} \frac{\pa}{\pa s_{21}}, \qquad \mathsf{R}_{\mathfrak{g}_{21}} = \frac{1}{s_{22}}  \frac{\pa}{\pa s_{21}}. \]
 The vector field $\mathsf{R}_{\mathfrak{g}_{22}}$ is the so-called radial vector field.
 \begin{prop}
 The algebra generated by the vector fields $\mathsf{R}$, $\mathsf{R}_{\mathfrak{g}_{22}}$, $\mathsf{R}_{\mathfrak{g}_{21}}$ is isomorphic to the $\mathfrak{sl}_2(\mathbb{C})$ Lie algebra. 
 \end{prop}
 \begin{proof} The proof follows from the explicit evaluation of the commutation relations. Explicitly, we find
\[ [\mathsf{R}, \mathsf{R}_{\mathfrak{g}_{21}}] = \mathsf{R}_{\mathfrak{g}_{22}}, \quad [\mathsf{R}, \mathsf{R}_{\mathfrak{g}_{22}}] = -2 \mathsf{R}, \quad [ \mathsf{R}_{\mathfrak{g}_{21}}, \mathsf{R}_{\mathfrak{g}_{22}}] = 2 \mathsf{R}_{\mathfrak{g}_{21}}. \]
\end{proof}
The construction of this section leads to version of a theorem of \cite{Movasati:20113} for the family \eqref{eq:ellfam}.

\begin{thm}
There is a ring isomorphism between the  ring of regular functions $\mathcal{O}_{\mathsf{T}}$ equipped with differential structures coming form the vector fields $\mathsf{R}$, $\mathsf{R}_{\mathfrak{g}_{22}}$, $\mathsf{R}_{\mathfrak{g}_{21}}$ and the differential ring $R_{1^*}$  with the differential structures described in the section  \ref{sec:app}.
\end{thm}
%\begin{thm}
%The ring of regular functions $\mathcal{O}_{\mathsf{T}}$ equipped with differential structures coming form the vector fields $\mathsf{R}$, $\mathsf{R}_{\mathfrak{g}_{22}}$, $\mathsf{R}_{\mathfrak{g}_{21}}$ is isomorphic as differential ring to the ring of quasi-modular forms on $\G_0(1)^*$ (see Appendix \ref{sec:app}) with the differential structures described in the section  \ref{sl2csection}. 
%\end{thm}
\begin{proof}
The isomorphism is given by sending $(z, s_{22}, s_{21})$ to 
\[ \alpha = 432 z, \qquad A= s_{22}, \qquad  E = ( 1 - 864 z) s_{22}^2 - 12 s_{21} s_{22}. \]
Then modular vector field $\mathsf{R}$ is equivalent to the differential relations from  \ref{sec:app}. Moreover, 
the action of $\mathsf{R}_{\mathfrak{g}_{22}}$ and $\mathsf{R}_{\mathfrak{g}_{21}}$ is equivalent to the action of $W$ and $\mathfrak{d}$ respectively. 
\end{proof}

\section{Non-compact CY varieties and mixed Hodge structure} \label{sec:nc}

\subsection{Local CY and mirror families}\label{ssec:lcy}
In this work we consider toric CY threefolds which are given by the total space of a bundle over a surface or a curve, we call these local CY threefolds. Toric varieties can be described by a fan $\Sigma$, but substantial amount of information about it is encapsulated by a charge matrix $Q^i_k$ that encodes the linear relations between one dimensional cones $\Sigma(1)$ of the fan $\Sigma$, i.~e.~$ \sum_{i=0}^{s-1} Q^i_k v_i = 0$ for $v_i \in \Sigma (1)$ and $s=|\Sigma(1)|$. The mirror families of such varieties are known to have the following form, see e.~g.~\cite{klemm1996self, katz1997mirror,chiang1999local,hosono2000local}: 
\be \qquad 
{\mathcal{X}} = \left\{uv + F_{\pmb{a}}(y_i)=  u v + \sum_{i=0}^{s-1}a_i y_i = 0  \; | ~ u, v \in \mathbb{C}, y_i \in \mathbb{C}^{*}, \prod_{i=0}^{s-1} y_i^{Q^i_k} = 1 \right\}. \label{mirror}  
\ee
It is also common to set as a mirror family a family of curves 
\be
\Sigma = \left\{ F_{\pmb{a}}(y_i)=  \sum_{i=0}^{s-1}a_i y_i = 0  \; | ~  y_i \in \mathbb{C}^{*}, \prod_{i=0}^{s-1} y_i^{Q^i_k} = 1 \right\}. \label{mirror2}  
\ee
We are interested in a special class of local CY varieties, namely canonical bundles over del Pezzo surfaces. In this case all the information can be encoded in a 2-dimensional reflexive polytope $\D$.  In particular, the set of the one dimensional cones $v_i$ of the fan $\Sigma$ corresponds to the set $A(\D)$ of the integral points $m_i$ of the polytope $\D$, thus the coordinates $y_i$ can be parametrized by 

\[ y_i=t^{m_i} \coloneqq t_1^{m_{i,1}} t_2^{m_{i,2}}, \quad t_1, t_2 \in \mathbb{C}^*,  \] 
for $m_i = (m_{i,1}, m_{i,2} ) \in A(\D)$. 
Hence the information about the general fibre $X_{\pmb{a}}$ of $\mathcal{X}$ is encoded in the polynomial 
\[ F_{\pmb{a}}(t_1, t_2) = \sum_{{m_i} \in A(\D)} a_i t^{m_i}. \]
The parameters $a_i$ redundantly describe the deformations of complex structure on $X_{\pmb{a}}$. The GIT quotient of the space of parameters $a_i$ by the natural torus action induced by the action $F(t_1, t_2)\rightarrow \l_0 F(\l_1 t_1, \l_2 t_2)$ gives the complex structure moduli space $\mathsf{B}$, for which the torus-invariant local coordinates can be chosen as $z_k=(-1)^{Q_k^0}\prod_i a_i^{Q^i_k}$. 

The holomorphic 3-form  on $X_{\pmb{a}}$ is given by the residue of the top form on the ambient space
\[ \Omega = \mathrm{Res}\left[\frac{1}{F_{\pmb{a}}(t_1,t_2) + u v}\frac{dt_1}{t_1} \wedge \frac{dt_2}{t_2}\wedge{du}\wedge{dv}\right], \label{3form} \]
where $\mathrm{Res}$ is the  residue map
\[\mathrm{Res}: H^4( \mathbb{C}^2 \times (\mathbb{C}^*)^2 \setminus X_{\pmb{a}}) \rightarrow  H^3(X_{\pmb{a}}),\]
which sends  an $\omega \in H^4( \mathbb{C}^2 \times (\mathbb{C}^*)^2 \setminus X_{\pmb{a}})$ to $  \int_\gamma w $, with $\gamma \in H_4 ( \mathbb{C}^2 \times (\mathbb{C}^*)^2\setminus X_{\pmb{a}})$ being a cycle around $ uv + F_{\pmb{a}}(t_1,t_2) = 0 $. The periods of the holomorphic 3-form satisfy Picard-Fuchs equations, which are obtained from a GKZ system \cite{Gel:1989} of differential equations
\[
\left[ \prod_{i: Q^i_k >0}\left( \frac{\p}{\p a_i} \right)^{Q^i_k } - \prod_{i: Q^i_k <0}\left( \frac{\p}{\p a_i} \right)^{-Q^i_k } \right] \int_\gamma \Omega =0, \quad \gamma \in H_3(X_{\pmb{a}},\mathbb{Z}). \label{eq:gkz}
\]

\subsection{Mixed Hodge structure}

In contrast to the compact CY varieties, cohomology groups of non-compact CY varieties are naturally endowed with a mixed Hodge structure introduced by \cite{deligne1971}.
\begin{dfn}
	A \textit{mixed Hodge structure} $(H_\mathbb{Z},W_\bullet,F^\bullet)$  is a $\mathbb{Z}$-module $H_\mathbb{Z}$ together with an increasing (weight) filtration $W_\bullet$ on $H_\mathbb{Q}=H_\mathbb{Z}\otimes_\mathbb{Z}\mathbb{Q}$ and a decreasing (Hodge) filtration $F^\bullet$, which defines a pure $\mathbb{Q}$-Hodge structure of weight $i$ on the graded piece $\operatorname{Gr}^W_i=W_i/W_{i-1}$.

A mixed Hodge structure  $(H_\mathbb{Z},W_\bullet,F^\bullet)$ is called \textit{graded-polarizable} if the induced Hodge structure on $\operatorname{Gr}_i^W$ is polarizable for all $i$. 
\end{dfn}
\begin{dfn}
	A \textit{variation of mixed Hodge structure} on a complex variety $\mathsf{B}$ is a triple $(\mathcal{H}_\mathbb{Z},\mathcal{W}_\bullet,\mathcal{F}^\bullet)$, where
	\begin{enumerate}
		\item $\mathcal{H}_\mathbb{Z}$ is a locally constant sheaf of finitely generated $\mathbb{Z}$-modules on $\mathsf{B}$,
		\item $\mathcal{W}_\bullet$ is an increasing filtration of $\mathcal{H}_\mathbb{Q}=\mathcal{H}_\mathbb{Z}\otimes_\mathbb{Z}\mathbb{Q}$ by locally constant subsheaves,
		\item $\mathcal{F}^\bullet$ is a decreasing filtration on $\mathcal{H}=\mathcal{H}_\mathbb{Z}\otimes_\mathbb{Z}\mathcal{O}_\mathsf{B}$ by holomorphic subbundles, satisfying Griffiths' transversality,
	\end{enumerate}
	such that for each $b \in \mathsf{B}$ the filtrations $\mathcal{F}_b^\bullet$ and $\mathcal{W}_{\bullet,b}$ define a mixed Hodge structure.  
	A variation of mixed Hodge structure is called \textit{graded-polarizable} if there exists a pairing $\langle-,-\rangle$, flat with respect to the Gauss-Manin connection $\nabla$ on each $\operatorname{Gr}^\mathcal{W}_i=\mathcal{W}_i/\mathcal{W}_{i-1}$.
\end{dfn}

A mixed Hodge structure on the middle cohomology $H^3(X_{\pmb{a}},\mathbb{C})$ for the families \eqref{mirror} was described in \cite{Batyrev:1994,konishi2010local} (see also \cite{stienstra1997resonant}). Define the following ring: 
\[ \mathbf{S}_\Delta  = \bigoplus_{k\geq0} \mathbf{S}_\Delta^k, \qquad \mathbf{S}_\Delta^k = \bigoplus_{m \in A(\Delta(k))} \mathbb{C} t_0^k t^m,  \] 
where  $t_0$ is an additional parameter and \[ \Delta(k)  \coloneqq \left\{ m \in \mathbb{Z}^2 \subseteq \mathbb{R}^2 \: | \:  \frac{m}{k} \in \D \right\} \]
for $k \geq 1$, $\D(0) \coloneqq \{ 0 \} \subset \mathbb{R}^2$. The grading is given by $ \operatorname{deg}(t_0^k t^m) = k$.  
Define also the following differential operators on $\mathbf{S}_\D$
\[  \mathcal{D}_0 \coloneqq \th_{t_0} + t_0 F_{\pmb{a}}, \quad  \mathcal{D}_i \coloneqq \th_{t_i} + t_0  \th_{t_i}F_{\pmb{a}} , \quad i=1,\dots,s-1,\] and a $\mathbb{C}$-vector space 
\[ \mathcal{R}_{F_{\mathbf{a}}} \coloneqq \mathbf{S}_\D / (\sum_{i=0}^{s-1} \mathcal{D}_i \mathbf{S}_\D). \] 
There is a decreasing filtration on $\mathcal{R}_{F_{\mathbf{a}}}$
\[  0 \subset \mathcal{E}^0 \subset \mathcal{E}^{-1} \subset \mathcal{E}^{-2} \subset \dots,  \]  
with $\mathcal{E}^{-k}$ being a subspace generated by monomials of degree $\leq k$. 

There is also an increasing filtration. Define $ I_\D^{(j)}$ to be the homogeneous ideals in $\mathbf{S}_\D$ generated as $\mathbb{C}$-subspaces by all monomials $t_0^k t^m$ ($k\geq 1$) with $m \in \D(k)$ which does not belong to any face of codimension $j$. Since everything belongs to the codimension 0 face, we get $ I_\D^{(0)}$ = 0. Everything that belongs to the face of codimension 2 also belongs to some face of codimension 1, therefore $ I_\D^{(1)} \subset I_\D^{(2)}$. There are no faces of codimension 3, thus $ I_\D^{(3)}  =\bigoplus_{k\geq1} \mathbf{S}_\Delta^k $, which contains the last two, and set $ I_\D^{(4)} = \mathbf{S}_\D$. These form an increasing filtration on $\mathbf{S}_\D$:
\[ 0 =  I_\D^{(0)} \subset I_\D^{(1)} \subset I_\D^{(2)} \subset I_\D^{(3)} \subset I_\D^{(4)} = \mathbf{S}_\D, \]
which defines under the quotient an increasing filtration on $\mathcal{R}_{F_{\pmb{a}}}$: 
 \[ 0 = \mathcal{I}_0 \subset \mathcal{I}_1 \subset \mathcal{I}_2 \subset \mathcal{I}_3 \subset \mathcal{I}_4 = \mathcal{R}_{F_{\pmb{a}}}, \]
where $\mathcal{I}_j$ is the image of $I_\D^{(j)}$ in $\mathcal{R}_{F_{\pmb{a}}}$. 

Let us consider $\mathbf{S}_\Delta [\pmb{a}] = \mathbf{S}_\Delta \otimes \mathbb{C}[\pmb{a}]$ and \[ \mathcal{R}_{F_{\pmb{a}}} [\pmb{a}] \coloneqq \mathbf{S}_\D [\pmb{a}] / (\sum_{i=0}^n \mathcal{D}_i \mathbf{S}_\D [\pmb{a}]). \] 
On $\mathbf{S}_\Delta [\pmb{a}]$ we can define the following differential operators:   
\[ \mathcal{D}_{a_i} \coloneqq \frac{\pa}{\pa a_i} + t_0 t^{m_i}.  \] 
Since they commute with $\mathcal{D}_i$, they descend to $\mathcal{R}_{F_{\pmb{a}}} [\pmb{a}]$. 
\begin{thm} [\cite{konishi2010local}]
There is an isomorphism 
\[  r: \mathcal{R}_{F_{\pmb{a}}} \cong H^3( X_{\pmb{a}}, \mathbb{C}),  \]
for every $\pmb{a}$. 
Therefore, as a consequence,  the filtrations $ \mathcal{E}_\bullet\,, \mathcal{I}_\bullet$ define a mixed Hodge structure on $H^3( X_{\pmb{a}},\mathbb{C})$, namely  $F^{6-k}=r(\mathcal{E}^{-k}) $ and $ W_3= r(\mathcal{I}_1)$, $W_4 = W_5 = r(\mathcal{I}_3)$, $W_6 = r(\mathcal{I}_4)$. 
Moreover the derivations $ \mathcal{D}_{a_i} $ correspond to the Gauss-Manin connection $\nabla_{\frac{\pa}{\pa a_i}}$.
\end{thm}
\begin{rem}[\cite{konishi2010local}]
The weight 3 filtration space $W_3$ is in fact just the first cohomology of the corresponding curve $\overline{\Sigma}_{\pmb{a}}$, which is the compactification of a fibre of \eqref{mirror2} at $\pmb{a}$:
\[ W_3 H^3(X_{\pmb{a}}) \cong H^1(\overline{\Sigma}_{\pmb{a}}). \] 
\end{rem}

For the families \eqref{mirror} there is an inclusion $W_3 H^3(X_{\pmb{a}}) \cong H^1(\overline{\Sigma}_{\pmb{a}}) \subset H^3(X_a)$. We define a polarization $\langle -,- \rangle $ to be equal to the intersection pairing on $W_3 H^3(X_a)$ and 0 otherwise. It was shown in \cite{konishi2010local} that this pairing is flat on $\operatorname{Gr}_i^W$ with respect to the Gauss-Manin connection. Since in the case of local del Pezzo all the mirror curves are elliptic curves ( see \cite{chiang1999local}), in a suitable basis the pairing matrix is
%\[\begin{pmatrix}
% 0 &\dots & 0 & 0 \\
% \vdots &\ddots & \vdots & \vdots \\
% 0 & \dots & 0 & -1 \\
% 0 & \dots & 1 & 0 
%\end{pmatrix} . \]
\[
\begin{pmatrix}
 0_{(b_n-2)\times(b_n-2)} & 0_{(b_n-2)\times 2}
  \\
  0_{2\times(b_n-2)} & \begin{matrix}
    0 & -1 \\
    1 & 0 
  \end{matrix}
\end{pmatrix}, \]
where $0_{a\times b}$ is an $a\times b$ block of zeroes.

\subsection{Enhanced non-compact CY varieties}

An important property of the intersection pairing in the case of compact CY varieties is that it is flat with respect to the Gauss-Manin connection. In the case of variation of mixed Hodge structure we no longer have this property, but we have flatness of the pairing on the graded components $\operatorname{Gr}^\mathcal{W}_i=\mathcal{W}_i/\mathcal{W}_{i-1}$. This fact motivates the following definition. 

\begin{dfn} \label{defenh}
An \textit{enhanced CY variety} of dimension $n$ is a pair $(X, \w)$, where $X$ is a CY variety of dimension $n$ and
 \[\w = ( \w_1,\dots ,  \w_i,\dots, \w_{{b_n}})^{\rm tr},\quad b_n =\dim H^n(X, \mathbb{C}),  \]
 is   a basis of $H^n(X, \mathbb{C})$ such that:
  \begin{itemize}
     \item  It respects the Hodge filtration, i.~e.
    \[  (\w_1, \ldots, \w_{\dim F^k}) \]
    spans $F^k$; 
     \item it respects the weight filtration, i.~e.
    \[   (\w_{b_n - \dim W_k+1}, \ldots, \w_{b_n}) \]
    spans $W_k$; 
    \item the pairing $\langle -, - \rangle$ on $\operatorname{Gr}^{W}_i$ takes the form of a constant matrix $\Phi_i$ in this basis for all $i$. 
 \end{itemize}   
\end{dfn}

The construction of the moduli space of enhanced varieties, as well as the Gauss-Manin Lie algebra generalize to this setting.  The difference is the additional constraints imposed on the entries of $S$ by the second condition above.  It sets some of the lower-diagonal blocks to zero, in a similar fashion as the Hodge filtration condition sets the upper-diagonal blocks of $S$ to zero. Moreover, the pairing condition of the Definition \ref{dfn:enhco} is imposed on each graded component of the weight filtration separately.
\section{Examples from mirror symmetry}
The construction of section \ref{sec:nc} is applied to two families of non-comact CY threefolds, mirrors to local $\mathbb{P}^2$ and local $\mathbb{F}_2$. We construct the rings of regular functions on $\mathsf{T}$ and the Gauss-Manin Lie algebra and prove Theorems \ref{thm:1}, \ref{thm:2} and \ref{thm:3}.
\subsection{Local $\mathbb{P}^2$}
\subsubsection{Setup} The first example we study is the total space of the canonical bundle over the projective plane ${\mathbb{P}^2}$. It is defined by the toric charge vectors $Q=(-3,1,1,1)$ and the mirror family is given by
\[
\mathcal{X} = \left\{   u v + a_0 y_0 + a_1 y_1 + a_2 y_2 + a_3 y_3 = 0  \; | ~ u, v \in \mathbb{C}, y_k \in \mathbb{C}^{*}, \frac{y_1 y_2 y_3}{y_0^3} = 1 \right\}.
\]
We have \[ F_{\pmb{a}}(t_1, t_2) = a_0 + a_1 t_1 + a_2 t_2 + \frac{a_3}{t_1 t_2}. \]
The torus invariant coordinate reads
\[
    z=-\frac{a_1a_2a_3}{a_0^3},
\]
and the Picard-Fuchs operator is
\[
    \mathcal{L}=(\theta^2-3z(3\theta+1)(3\theta+2))\theta,\quad \theta=z\frac{d}{dz}.
\]
The middle cohomology can be written as 
\[H^3(X_{z}) \cong \mathcal{R}_{F_{\pmb{a}}} = \mathbb{C} 1 \oplus \mathbb{C} t_0  \oplus \mathbb{C} t_0^2, \]
where we denote the general fibre of $\mathcal{X}$ by $X_z$ to emphasize the dependence on $z$. The mixed Hodge structure described in the previous section is 
\[
 0={W}_2\subset {W}_3 ={W}_4 = {W}_5 = \mathbb{C}  t_0 \oplus \mathbb{C}  t_0^2 \subset {W}_6 = \mathcal{R}_{F_{\pmb{a}}},
\]
\[
0={F}^4 \subset {F}^3 = \mathbb{C}  1 \subset {F}^2 =  \mathbb{C}  1 \oplus \mathbb{C}  t_0 \subset {F}^1 = {F}^0 = \mathcal{R}_{F_{\pmb{a}}}.
\]
We can generate it with the help of $\mathcal{D}_{a_0}$:
\[ \mathcal{D}_{a_0} (1) = \left( \frac{\pa}{\pa a_0} + t_0 \right) (1) = t_0, \qquad \mathcal{D}_{a_0} (t_0) = t_0^2.  \]
There is a patch of the moduli space $\mathsf{B}$, where  \[\th_{a_0} = -3\th = -3 z \frac{d}{d z} = - 3 \th, \] holds. Thus we can take $\left( \Omega, \nabla_\th \Omega, \nabla^2_\th \Omega \right)$ as a basis of $H^3(X_z)$ that satisfies first two conditions from the Definition \ref{defenh}. The only non-zero pairing is between $\nabla_\th \Omega$ and $\nabla^2_\th \Omega$. It is the Yukawa coupling $Y_{111}$ for local $\mathbb{P}^2$, namely
\[ \langle \nabla^2_\th \Omega, \nabla_\th \Omega \rangle  =  Y_{111}= -\frac{1}{3(1-27z)}. \] 
In the basis 
\[ \w = \left( \Omega, \nabla_\th \Omega,  Y_{111}^{-1}\nabla^2_\th \Omega \right), \]
the pairing has the form 
\[ \begin{pmatrix}
 0 &0&0\\
 0&0&-1\\ 
 0&1&0 
\end{pmatrix}.  \]

\subsubsection{Moduli space $\mathsf{T}$}
Let us fix  $\Phi_3 = \begin{pmatrix}
 0 & -1 \\
 1 &0 
\end{pmatrix}$ and $\Phi_6 = \begin{pmatrix}
 0
\end{pmatrix}$. 
We saw in the previous subsection that $\w$ is a basis satisfying the conditions of the Definition \ref{defenh}, therefore we can construct (locally) the moduli space $\mathsf{T}$ by considering the complex structure modulus $z$ and algebraically independent entries of the matrix $S$. 
The requirement on $S$ to preserve the Hodge filtration restricts it to the lower-diagonal form 
\[ S = \begin{pmatrix}
 s_{11} & 0 & 0 \\
 s_{21} & s_{22} & 0 \\ 
 s_{31} & s_{32} & s_{33} 
\end{pmatrix}.  \] 
The second condition -- preservation of the weight filtration, sets $s_{21} = s_{31} = 0$. 
Next, we have to satisfy the condition that \[
\langle -, - \rangle \rvert_{\operatorname{Gr}^W_3} = \Phi_3, \qquad \langle -, - \rangle \rvert_{\operatorname{Gr}^W_6} = \Phi_6,  \] 
note that $\operatorname{Gr}^W_4$ and $\operatorname{Gr}^W_5$ are empty.  The second condition is empty, and the first one implies $s_{22} = s_{33}^{-1}$, thus we have 
\[S = \begin{pmatrix}1&0&0\\ 0 & s_{33}^{-1} &0\\ 0&s_{32}&s_{33}\end{pmatrix}. \]
The element $s_{11}$ corresponds to the normalization of the holomorphic $3$-form. One can check that it is constant with respect to the modular vector field and we set it to be $1$ for simplicity. 
%, as it decouples from other $s_{ij}$.
\subsubsection{Modular vector field and Gauss-Manin Lie algebra}
Modular vector fields have the form 
\[ \nabla_{\mathsf{R}} \wt  = \begin{pmatrix} 0 & * & 0 \\
0 & 0 & * \\ 
0 &0 & 0 \end{pmatrix} \wt. \] 
We make a choice of normalization inspired by \cite{alim2014special} 
\be \nabla_{\mathsf{R}} \wt  = \begin{pmatrix} 0 &  Y^{-1}(\mathsf{t}) & 0 \\
0 & 0 & 1 \\ 
0 &0 & 0 \end{pmatrix} \wt.  \label{prop4one}\ee 
The Gauss-Manin connection matrix in the basis $\w$ reads
\[
    \nabla_{\theta}\omega=\begin{pmatrix}0&1&0\\ 0&0&Y_{111}\\ 0&18z &0 \end{pmatrix}\omega = A_\th \w .
\]
Let  \[ \mathsf{R} = i_z \frac{\pa}{\pa z } + i_{32} \frac{\pa}{\pa s_{32}} + i_{33} \frac{\pa}{\pa s_{33 }}, \] 
then, \be \nabla_{\mathsf{R}} \wt  =\left[ \frac{i_z}{z} ( S A_\th S^{-1}) + i_{32} \frac{\pa S}{\pa s_{32}} S^{-1} + i_{33} \frac{\pa S}{\pa s_{33}} S^{-1} \right] \wt. \label{prop4two} \ee
\begin{prop}\label{prop:mvf}
There exists a unique modular vector field $\mathsf{R}$ on $\mathsf{T}$. 
It is given by
\[
    \mathsf{R}=\frac{z s_{33}^2}{Y_{111}}\frac{\partial}{\partial z} - s_{32} s_{33}^2\frac{\partial}{\partial s_{33}}+\frac{18z s_{33}^3}{Y_{111}}\frac{\partial}{\partial s_{32}}. 
\]
\end{prop}
\begin{proof}
Using non-degeneracy of the Gauss-Manin connection we equate the right hand sides of the equations \eqref{prop4one} and \eqref{prop4two}. In such a way one obtains a system of linear equations that has a unique solution for $(i_z, i_{32}, i_{33})$.
\end{proof}
%\begin{proof}
%To prove the claim, let us first consider another moduli space $\widetilde{\mathsf{T}}$, which contains $\mathsf{T}$, by assuming that the entries $s_{ij}$ are independent parameters,~i.e. we relax the condition on the pairing in Definition \ref{defenh}. Introduce $\dot{f} = df(\mathsf{R})$ for $f\in \mathcal{O}_{\mathsf{T}}$, then the existence of $\mathsf{R}$ is equivalent to the following set of differential equations, by \eqref{eq:mvfco}
%\[
%\begin{cases}
%    \dot{z} = {z s_{33}^2}{Y_{111}^{-1}}, \\
%    \dot{s}_{32}={18z s_{33}^3}{Y_{111}^{-1}},\\
%    \dot{s}_{33} = - s_{32} s_{33}^2.
%  \end{cases}
%\]
%which define $\dot{s}_{ij}$ and give a differential structure on the ring $\mathbb{C}[z,s_{ij}]$. The rest of the relations are trivial.

%Secondly, we consider the full moduli space $\mathsf{T}$. Consider the map
%\begin{equation*}
%    \widetilde{\mathsf{T}}\rightarrow {\rm Mat}_{b_3\times b_3}(\mathbb{C}),\quad (z, S)\mapsto SQ S^{\rm tr},
%\end{equation*}
%where $Q$ is the paining $\langle-,-\rangle$ in some basis of $H^3(X_z)$ and $b_3={\rm dim}(H^3(X_z))$. Then $\mathsf{T}$ is the fibre of this map at the point of constant pairing $\Phi$. It is enough to prove that $\mathsf{R}$ is tangent to the above map at this point, which is easy to see (e.g. proof of Theorem 1.2 in \cite{Nikdelan:2017}). Uniqueness follows from non-degeneracy of the Gauss-Manin connection, following similar arguments as \cite{Alim:2016}.
%\end{proof}
Introduce $\dot{f} = df(\mathsf{R})$ for $f\in \mathcal{O}_{\mathsf{T}}$, then the existence of $\mathsf{R}$ is equivalent to the following differential structure on the ring $\mathcal{O}_{\mathsf{T}}$:
\[
\begin{cases}
    \dot{z} = {z s_{33}^2}{Y_{111}^{-1}}, \\
    \dot{s}_{32}={18z s_{33}^3}{Y_{111}^{-1}},\\
    \dot{s}_{33} = - s_{32} s_{33}^2\,.
  \end{cases}
\]

\begin{customthm}{1}
\label{prop:P2}
There is an isomorphism between the ring of regular functions $\mathcal{O}_{\mathsf{T}}$ equipped with the differential structure $\mathsf{R}$ and the ring $R_3$ of quasi-modular forms on $\Gamma_0(3)$ with the derivation $\pa_\tau$. 
\end{customthm}
\begin{proof}
Define \[ A=\sqrt{-3 }s_{33}, \qquad \alpha = 27 z,  \]  
\[ B=(1-27z)^{1/3}A,\quad  C=3z^{1/3} A, \quad  E= d{\log (B^3 C^3 )(\mathsf{R})},  \]
Then we can rewrite the differential  ring as 
\[\label{eq:ring}
\begin{split}
\dot{\alpha}&=\alpha(1-\alpha)A^2,\\
\dot{A}&=\frac{1}{6} \left( A E+C^3 - B^3\right),\\
\dot{B}&=\frac{1}{6}B(E-A^2),\\
\dot{C}&=\frac{1}{6}C(E+A^2),\\
\dot{E}&=\frac{1}{6}(E^2-A^4),
\end{split}
\]
which is the differential ring $R_3$ of quasi-modular forms on  $\Gamma_0(3)$ in \ref{sec:app}.
\end{proof}

\subsubsection{Gauss-Manin Lie algebra}
\begin{customthm}{3}[Part 1]
\label{thm:GMLAP2}
The Gauss-Manin Lie algebra for local $\mathbb{P}^2$ is spanned by
\[
    \begin{split}
        &\mathsf{R}=\frac{z s_{33}^2}{Y_{111}}\frac{\partial}{\partial z}- s_{32} s_{33}^2 \frac{\partial}{\partial s_{33}}+\frac{18z s_{33}^3}{Y_{111}}\frac{\partial}{\partial s_{32}},\\
        &\mathsf{R}_{\mathfrak{g}_{33}}=s_{33}\frac{\partial}{\partial s_{33}}+s_{32}\frac{\partial}{\partial s_{32}},\\
        &\mathsf{R}_{\mathfrak{g}_{32}}=s_{33}\frac{\partial}{\partial s_{32}},
    \end{split}
\]
and it is isomorphic to $\mathfrak{sl}_2(\mathbb{C})$.
\end{customthm}
\begin{proof} $\operatorname{Lie}(\mathsf{G})$ is generated by \[ \mathfrak{g}_{33} = \left. \left(\frac{\pa}{\pa s_{33}} S \right) S^{-1} \right\vert_{S=\operatorname{Id}}= \begin{pmatrix}
  0&0&0\\
  0&-1 & 0 \\
  0& 0 & 1 
  \end{pmatrix}, \]
  \[ \mathsf{g}_{21} = \left. \left(\frac{\pa}{\pa s_{32}} S \right) S^{-1} \right|_{S=\operatorname{Id}}= \begin{pmatrix}
  0&0&0\\
  0&0 & 0 \\
  0& 1 & 0 
  \end{pmatrix}. \]
The corresponding vector fields are exactly $\mathsf{R}_{\mathsf{g}_{33}}$ and $\mathsf{R}_{\mathsf{g}_{32}}$. 
It can be checked explicitly that commutators are
\[ [\mathsf{R}, \mathsf{R}_{\mathfrak{g}_{32}}] = \mathsf{R}_{\mathfrak{g}_{33}}, \quad [\mathsf{R}, \mathsf{R}_{\mathfrak{g}_{33}}] = -2 \mathsf{R}, \quad [ \mathsf{R}_{\mathfrak{g}_{32}}, \mathsf{R}_{\mathfrak{g}_{33}}] = 2 \mathsf{R}_{\mathfrak{g}_{32}}. \]
\end{proof}

\subsection{Local $\mathbb{F}_2$} The canonical bundle over the second Hirzebruch surface $K_{\mathbb{F}_2}$ is defined by the toric charge vectors $Q_1=(-2,1,0,1,0)$ and $Q_2=(0,0,1,-2,1)$ and the mirror family is given by
\[
\mathcal{X} = \left\{   u v + \sum_{i=0}^4 a_i y_i = 0  \; | ~ u, v \in \mathbb{C}, y_k \in \mathbb{C}^{*}, \frac{y_1 y_3}{y_0^2} = 1, \frac{y_2 y_4}{y_3^2} = 1\right\}.
\]
We have 
\[F_{\pmb{a}}(t_1,t_2) = a_0 + a_1 t_1 + a_2 t_2 + \frac{a_3}{t_1} + \frac{a_4}{t_1^2 t_2}.\]
The torus-invariant coordinates are 
\[
z_1=\frac{a_1a_3}{a_0^2}\,,\quad z_2=\frac{a_2 a_4}{a_3^2}\,,
\]
and the Picard-Fuchs operators read
\begin{eqnarray*} \label{pff2}
\mathcal{L}_1&=& \theta_1 (\theta_1-2\theta_2)-2 z_1 (2\theta_1+1)\theta_1\,, \\
\mathcal{L}_2&=&\theta_2^2-z_2 (\theta_1-2\theta_2-1)(\theta_1-2\theta_2)\,,\quad \theta_i=z_i\frac{d}{dz_i}.
\end{eqnarray*}
The middle cohomology can be written as 
 \[ H^3(X_z) \cong \mathcal{R}_{F_{\pmb{a}}} = \mathbb{C} 1 \oplus \mathbb{C} t_0 \oplus \mathbb{C} \frac{t_0}{t_1} \oplus \mathbb{C} t_0^2.\]
The mixed Hodge structure is
\[ 0 = W_2 \subset W_3 = \mathbb{C}t_0 \oplus \mathbb{C} t_0^2 \subset W_4 = W_5 = W_3 \oplus \mathbb{C} \frac{t_0}{t_1}  \subset W_6 = \mathcal{R}_{F_{\pmb{a}}},\]
\[
0=F^4 \subset {F}^3 = \mathbb{C}  1 \subset {F}^2 =  \mathbb{C}  1 \oplus \mathbb{C}  t_0 \oplus \mathbb{C} \frac{t_0}{t_1} \subset {F}^1 = {F}^0 = \mathcal{R}_{F_{\pmb{a}}}. \]
Using the Gauss-Manin connection we obtain 
\[  \mathcal{D}_{a_0} (1) = t_0, \qquad \mathcal{D}_{a_0} (t_0) = t_0^2, \qquad \mathcal{D}_{a_3} (1) = \frac{t_0}{t_1} .   \]
On the moduli space $\mathsf{B}$ there is a patch where 
\[ a_0 \frac{\pa}{\pa a_0} = -2 \th_1, \qquad a_3 \frac{\pa}{\pa a_3} = \th_1 - 2 \th_2, \]
therefore in the basis
\[ \left( \Omega, \nabla_{(\th_1 - 2\th_2)} \Omega, \nabla_{\th_1}\Omega, \nabla^2_{\th_1} \Omega \right), \] 
the pairing has the form 
\[\begin{pmatrix}
     0&0&0&0\\
     0&0&0&0\\
     0&0&0&-Y_{111}\\
     0&0&Y_{111}&0
    \end{pmatrix},\]
where $Y_{111}$ is the Yukawa coupling
\[Y_{111}= \frac{1}{(1-4z_1)^2-64z_1^2z_2}.\]

\subsubsection{Moduli space}
Let us fix 
\[ \Phi_3 = \begin{pmatrix} 0 & -1 \\ 1 & 0 \end{pmatrix}, \quad \Phi_4 = \begin{pmatrix} 0 \end{pmatrix}, \quad \Phi_6 = \begin{pmatrix} 0 \end{pmatrix}. \]
In addition to the complex structure moduli $z_1$ and $z_2$ we need to identify the independent elements of the matrix $S$. 
The Hodge filtration condition restricts it to the lower-block diagonal form 
 \[ S = \begin{pmatrix}
  s_{11} & 0&0&0 \\
  s_{21} &s_{22} & s_{23} & 0\\
  s_{31} & s_{32}& s_{33} & 0 \\ 
  s_{41} & s_{42} & s_{43} & s_{44} \\ 
 \end{pmatrix}. \]
The weight filtration condition gives 
 \[ S = \begin{pmatrix}
  s_{11} & 0&0&0 \\
  0 &s_{22} & s_{23} & 0\\
 0 & 0& s_{33} & 0 \\ 
 0 & 0 & s_{43} & s_{44} \\ 
 \end{pmatrix}. \]
The conditions on the pairing on $\operatorname{Gr}_4^W$ and $\operatorname{Gr}_6^W$ are empty and the corresponding condition on $\operatorname{Gr}_3^W$ gives us the final form of $S$:
 \be\label{eq:sF2} S = \begin{pmatrix}
  s_{11} & 0&0&0 \\
  0 &s_{22} & s_{23} & 0\\
 0 & 0& s_{44}^{-1} & 0 \\ 
 0 & 0 & s_{43} & s_{44} \\ 
 \end{pmatrix}. \ee
Moreover, we  again set $s_{11}=1$. 
\subsubsection{Modular vector fields}
We define modular vector fields by 
\begin{equation}\label{eq:R1} \nabla_{\mathsf{R}_1} \wt =\begin{pmatrix}
 0 &  Y^{-1}_1(\mathsf{t}) &  Y^{-1}_2(\mathsf{t}) & 0 \\
 0 &0& 0&   0\\
 0 & 0 & 0 & 1\\
 0  & 0 &  0 & 0 
 \end{pmatrix} \wt,\end{equation}
 and
 \begin{equation}\label{eq:R2}
 \nabla_{\mathsf{R}_2} \wt =\begin{pmatrix}
 0 & Y^{-1}_3(\mathsf{t}) & Y^{-1}_4(\mathsf{t}) & 0 \\
 0 &0& 0&   1\\
 0 & 0 & 0 & 0\\
 0  & 0 &  0 & 0 
 \end{pmatrix} \wt, 
 \end{equation}
 for some $Y_i(\mathsf{t}) \in \mathcal{O}_{\mathsf{T}}$.

\begin{thm}
There exist unique modular vector fields $\mathsf{R}_1,\mathsf{R}_2\in\mathfrak{X}(\mathsf{T})$ satisfying \eqref{eq:R1} and \eqref{eq:R2}. 
\end{thm}

The proof is analogous to the proof of Proposition \ref{prop:mvf}.

The existence of the modular vector fields is equivalent to two differential structures on $\mathcal{O}_{\mathsf{T}}$. By defining \[\p_{\tau_1}f \coloneqq df(\mathsf{R}_1)\quad \text{and}\quad \p_{\tau_2}f \coloneqq df(\mathsf{R}_2) \] 
for $f \in \mathcal{O}_{\mathsf{T}}$, we can write
\begin{equation}
\begin{split} 
\p_{\tau_1}\: z_1 =& z_1 s_{44}^2 s_{22}^{-1}\left( s_{22} + (1- 4 z_1)(1 - 4 z_2) s_{23} - 4 z_1 (1 + 4 z_2) s_{22} \right), \\
\p_{\tau_1}\; z_2 =& 2 z_2 (-1 + 4 z_2) s_{44}^2 s_{22}^{-1} \left( s_{23} + 4 z_1 s_{22} \right), \\
\p_{\tau_1} s_{22} =& 4 z_2 s_{44}^2  \left( s_{23}   +4 z_1 s_{22}  \right), \\
\p_{\tau_1} s_{23} =& 2 z_1 s_{44}^2 s_{22}^{-1} \left( (s_{22} + s_{23})^2 - 4 z_2 s_{23}^2  \right), \\
\p_{\tau_1} s_{43} =&  -2 z_1 s_{44}^2 s_{22}^{-1} \left( (-1+ 4z_2) s_{23} ( - s_{43}+ 4 z_2 s_{44})  \right.  \\ & \left. + s_{22} \left( s_{43} - 4 z_2 s_{43} + 4 z_2 (1-10 z_1 + 8 z_1 z_2) s_{44} \right) \right),  \\
\p_{\tau_1} s_{44}  =& - s_{44}^2 s_{22}^{-1} \left( 2 z_1 ( -1 + 4 z_2) s_{23} s_{44} + s_{22} \left( s_{43} + (-1 + 4 z_2) s_{44} \right) ) \right), \\
\p_{\tau_2} \: z_1 =& z_1 ( 1 - 4 z_1) ( 1 - 4 z_2) s_{44} s_{22}^{-1},  \\
\p_{\tau_2} \: z_2 =& 2 z_2 (1 - 4 z_2)s_{44} s_{22}^{-1},  \\
\p_{\tau_2} s_{22} =&  - 4 z_2 s_{44}, \\
\p_{\tau_2} s_{23} =& s_{43} + 2 z_1 s_{44} +2 z_1 s_{23} s_{44} s_{22}^{-1} + 8 z_1 z_2  s_{44}( 1 - s_{23} s_{22}^{-1}), \\
\p_{\tau_1} s_{43} =& 2 z_1 (1 - 4 z_2) s_{44} s_{22}^{-1} (s_{43} - 4 z_2 s_{44},\\
\p_{\tau_2} s_{44} =& 2 z_1 (-1 + 4 z_2) s_{44}^2 s_{22}^{-1} .
\end{split}
\end{equation}
\begin{customthm}{2} \label{prop:F2}
The differential ring $ \left( \mathcal{O}_{\mathsf{T}}, \pa_{\tau_1}, \pa_{\tau_2} \right)$ contains as differential subring the ring $R_2$ of quasi-modular forms on $\G_0(2)$ (see  \ref{sec:app}). 
\end{customthm}
\begin{proof}
By introducing a new variable
\[ u = \frac{64 z_1^2 z_2}{(1 - z_1)^2}, \]
and  
\[ A = \sqrt{2(1-4 z_1)} s_{44}, \qquad B = (1-u)^{1/4} A,\qquad C = u^{1/4} A, \] 
\[ E = \p_{\tau_1}\log(B^4 C^4) = \frac{2 - 8 s_{33} s_{43} - 8 z_1(1+8 z_2)}{s_{33}^2}, \] 
the differential ring relations can be rewritten as
\[\begin{aligned} \p_{\tau_1} u &= u(1-u) A^2, \\
\p_{\tau_1} A &= \frac{1}{8} A( E + \frac{C^4 - B^4}{A^2}), \\ 
\p_{\tau_1} B &= \frac{1}{8} B (E - A^2), \\ 
\p_{\tau_1} C &= \frac{1}{8} C (E + A^2), \\
\p_{\tau_1} E &= \frac{1}{8} (E^2 - A^4), \end{aligned}\]
which is the differential ring $R_2$. Furthermore, one can check that all of the functions $u, A,B,C,E$ are independent of $\tau_2$.
\end{proof}
\subsubsection{Gauss-Manin Lie algebra}
\begin{customthm}{3}[Part 2]\label{thm:GMLAF2}
The Gauss-Manin Lie algebra for local $\mathbb{F}_2$ is isomorphic to the semi-direct product $\mathfrak{L}_V \rtimes   \mathfrak{sl}_2(\mathbb{C})$, where $\mathfrak{L}_V$ denotes the Lie algebra of type V, as classified in \cite{Bianchi}, corresponding to an ideal of $\mathfrak{G}$ generated by $\mathsf{R}_{2}$, $\mathsf{R}_{\mathfrak{g}_{22}}$ and $\mathsf{R}_{\mathfrak{g}_{23}}$.
\end{customthm}
\begin{proof}
In this case $\operatorname{Lie}(\mathsf{G})$ is generated by 
\[ \mathfrak{g}_{22} = \left. \left(\frac{\pa}{\pa s_{22}} S \right) S^{-1} \right|_{S=\operatorname{Id}}= \begin{pmatrix}
0 & 0 & 0&0 \\
0&1&0&0\\
0 & 0 & 0&0 \\
0 & 0 & 0&0 \\
\end{pmatrix}, \] 
\[ \mathfrak{g}_{23} = \left. \left(\frac{\pa}{\pa s_{23}} S \right) S^{-1} \right|_{S=\operatorname{Id}}= \begin{pmatrix}
0 & 0 & 0&0 \\
0&0&1&0\\
0 & 0 & 0&0 \\
0 & 0 & 0&0 \\
\end{pmatrix}, \] 
\[ \mathfrak{g}_{43} = \left. \left(\frac{\pa}{\pa s_{43}} S \right) S^{-1} \right|_{S=\operatorname{Id}}= \begin{pmatrix}
0 & 0 & 0&0 \\
0 & 0 & 0&0 \\
0 & 0 & 0&0 \\
0&0&1&0\\
\end{pmatrix}, \] 
\[ \mathfrak{g}_{44} = \left. \left(\frac{\pa}{\pa s_{44}} S \right) S^{-1} \right|_{S=\operatorname{Id}}= \begin{pmatrix}
0 & 0 & 0&0 \\
0&0&0&0\\
0 & 0 & -1&0 \\
0 & 0 & 0&1 \\
\end{pmatrix}.  \] 
The associated vector fields $\mathsf{R}_{\mathfrak{g}_{ij}}$ on $\mathsf{T}$ defined by $  \nabla_{\mathsf{R}_{\mathfrak{g}_{ij}}} \wt = \mathfrak{g}_{ij} \wt $ are 
\[
\begin{split}
\mathsf{R}_{\mathfrak{g}_{22}} &= s_{22} \frac{\pa}{\pa s_{22}}+s_{23} \frac{\pa}{\pa s_{23}},  \\
\mathsf{R}_{\mathfrak{g}_{23}} &= s_{44}^{-1} \frac{\pa}{\pa s_{23}},  \\
\mathsf{R}_{\mathfrak{g}_{43}} &= s_{44}^{-1} \frac{\pa}{\pa s_{43}},  \\
\mathsf{R}_{\mathfrak{g}_{44}} &= s_{44} \frac{\pa}{\pa s_{44}}+s_{43} \frac{\pa}{\pa s_{43}}.
\end{split}
\]
Together with $\mathsf{R}_1$ and $\mathsf{R}_2$ these vector fields form the Gauss-Manin Lie algebra $\mathfrak{G}$ and the commutation relations are the following: 
\begin{align*}
    &[ \mathsf{R}_1, \mathsf{R}_{\mathfrak{g}_{44}}] = -2 \mathsf{R}_1, & &[ \mathsf{R}_1, \mathsf{R}_{\mathfrak{g}_{43}}] = \mathsf{R}_{\mathfrak{g}_{44}}, & &[\mathsf{R}_{\mathfrak{g}_{43}},\mathsf{R}_{\mathfrak{g}_{44}}] = 2 \mathsf{R}_{\mathfrak{g}_{43}}, \\ 
    &[ \mathsf{R}_1,\mathsf{R}_2] = 0, & &[ \mathsf{R}_1,\mathsf{R}_{\mathfrak{g}_{23}}] =\mathsf{R}_2, & &[\mathsf{R}_1, \mathsf{R}_{\mathfrak{g}_{22}}] = 0, \\ 
    &[ \mathsf{R}_{\mathfrak{g}_{43}}, \mathsf{R}_2]  = \mathsf{R}_{\mathfrak{g}_{23}}, & &[\mathsf{R}_{\mathfrak{g}_{43}}, \mathsf{R}_{\mathfrak{g}_{23}}]=0, & &[\mathsf{R}_{\mathfrak{g}_{43}},\mathsf{R}_{\mathfrak{g}_{22}}] = 0, \\ 
     &[ \mathsf{R}_{\mathfrak{g}_{44}}, \mathsf{R}_2] = \mathsf{R}_2, & & [\mathsf{R}_{\mathfrak{g}_{44}}, \mathsf{R}_{\mathfrak{g}_{23}}] =  - \mathsf{R}_{\mathfrak{g}_{23}}, & & [ \mathsf{R}_{\mathfrak{g}_{44}}, \mathsf{R}_{\mathfrak{g}_{22}}] = 0, \\
     &[ \mathsf{R}_2, \mathsf{R}_{\mathfrak{g}_{23}}]= 0, & &[\mathsf{R}_2, \mathsf{R}_{\mathfrak{g}_{22}}] = \mathsf{R}_2, & &[\mathsf{R}_{\mathfrak{g}_{23}}, \mathsf{R}_{\mathfrak{g}_{22}}] = \mathsf{R}_{\mathfrak{g}_{23}} . 
\end{align*}
The first line above is the $\mathfrak{sl}_2(\mathbb{C})$ Lie algebra corresponding to the quasi-modular forms. The last line is the 3-dimensional algebra $\mathfrak{L}_V$ according to the Bianchi's classification and it is an ideal in the Gauss-Manin Lie algebra  $\mathfrak{G}$.  Therefore, we have  \[\mathfrak{G}  \cong \mathfrak{L}_V \rtimes \mathfrak{sl}_2(\mathbb{C}) .  \] 
\end{proof}

\section{Conclusion}
In this article we studied moduli spaces of mirror non-compact CY threefolds enhanced with choices of differential forms in their middle cohomology. Certain vector fields contracted with the Gauss-Manin connection give rise in this context to a Lie-algebra which was put forward in \cite{Alim:2016} for arbitrary CY threefolds, generalizing an earlier similar construction for the elliptic curve \cite{Movasati:20113} and the mirror quintic \cite{Movasati:20112}.

In particular the resulting Lie algebra in the case of elliptic curves \cite{Movasati:20113} is isomorphic to the $\mathfrak{sl}_2 (\mathbb{C})$ Lie algebra of derivations associated to quasi-modular forms. The appearance of this Lie algebra as a Lie subalgebra of the full Gauss--Manin Lie algebra for families of CY threefolds is often connected to the appearance of classical quasi-modular forms. An $\mathfrak{sl}_2(\mathbb{C})$ Lie subalgebra was for example studied in the context of elliptically fibered CY manifolds \cite{Haghighat:2015}, where the modularity is originating from the elliptic fibers. In our work we explored a different source of classical modular objects\footnote{We use the terminology classical modular objects to refer to modular structures which have been widely studied, such as quasi-modular forms, Siegel and Jacobi forms. For generic compact CY manifolds the analogues of the modular structures are referred to as CY modular forms in \cite{Movasati:2017} and are in general not expressible in terms of classical modular forms.} which appear in the moduli spaces of CY geometries, namely the case of the mirror families of non-compact CY threefolds. The appearance of classical modular forms in this context is related to the mirror curve.

We focused on the cases where the mirror curves of the mirror geometries are of genus one and studied the enhanced moduli spaces and the Gauss-Manin connection of the variation of mixed Hodge structure. We identified the $\mathfrak{sl}_2(\mathbb{C})$ Lie subalgebra corresponding to the algebra of derivations of quasi-modular forms in the cases of the mirrors of local $\mathbb{P}^2$ and local $\mathbb{F}_2$. We note that the mirror geometry of local $\mathbb{F}_2$ has two parameters, although the moduli space is only one-dimensional, this was already addressed, for instance, in \cite{Aganagic:2008} and is related to the fact that $\mathbb{F}_2$ has two distinct curve classes but only one non-trivial four-cycle. On the level of the Gauss-Manin Lie algebra we note that the additional parameter gives rise to an additional vector field, enlarging the $\mathfrak{sl}_2(\mathbb{C})$ algebra of derivations of the quasi-modular forms in this case. It would be interesting to find an interpretation in the context of modular forms for this enhancement.

It is natural to expect that the study of the analogous moduli spaces of enhanced mirror CY families where the mirror curve is of genus higher than one would give rise to algebraic structures associated to modular objects associated to higher genus curves. The GMCD program for certain genus two curves was recently studied in \cite{cao2019gauss} giving rise to differential rings of Siegel modular forms and their associated algebraic structures. Differential rings attached to genus two curves were studied in \cite{klemm2015direct,ruan2019genus} in the context of mirror symmetry, it would be interesting to recover the corresponding Gauss-Manin Lie algebras in these cases.

\newcommand{\etalchar}[1]{$^{#1}$}

%\bibliographystyle{alpha}
%\bibliography{references}

\end{document}